\documentclass[10pt,a4paper]{amsart}

\usepackage{hyperref}
\usepackage{ccicons}

\hypersetup{
	colorlinks=true,
	linktocpage=true,
	linkcolor=blue,
	citecolor=blue,
	pdftitle={Uniformizable foliated projective structures},
	pdfauthor={Bertrand Deroin and Adolfo Guillot}
}

\newcommand{\del}[1]{\frac{\partial}{\partial #1}}
\newcommand{\indel}[1]{\partial/\partial #1}

\newcommand{\NN}{\mathbb{N}}
\newcommand{\ZZ}{\mathbb{Z}}
\newcommand{\CC}{\mathbb{C}}
\newcommand{\RR}{\mathbb{R}}
\newcommand{\HH}{\mathbb{H}}
\newcommand{\QQ}{\mathbb{Q}}

\newcommand{\PP}{\mathbb{P}}
\newcommand{\BB} {\mathbb{B}} 

\newcommand{\dd} {\mathrm{d}} 
\newcommand{\ee} {\mathrm{e}} 
\newcommand{\ii} {\mathrm{i}} 
 
\newtheorem{thm}{Theorem}
\newtheorem{lemma}[thm]{Lemma}
\newtheorem{prop}[thm]{Proposition}
\newtheorem{corollary}[thm]{Corollary}
 
\theoremstyle{definition}

\theoremstyle{remark}
\newtheorem{rmk}[thm]{Remark}
\newtheorem{example}[thm]{Example}

\numberwithin{equation}{section}
 
\begin{document}

\title[Uniformizable foliated projective structures]{Uniformizable foliated projective structures along singular foliations}

\author{Bertrand Deroin}
\address{CNRS-Laboratoire AGM-CY, Cergy Paris Université, 2 Avenue Adolphe Chauvin,
	95302, Cergy-Pontoise, France}
\email{bertrand.deroin@cyu.fr}

\author{Adolfo Guillot}
\address{Instituto de Matem\'aticas, Universidad Nacional Aut\'onoma de M\'exico, Ciudad Universitaria  04510,
Ciudad de M\'exico \\  Mexico}
\email{adolfo.guillot@im.unam.mx}




\thanks{The authors benefited from the support of the IRL 2001 \emph{Solomon Lefschetz} CNRS-UNAM, and of CY Initiative (ANR-16-IDEX-0008).}


\thanks{Published as: \textsc{Deroin, B.} and \textsc{Guillot, A.} Uniformizable foliated projective structures, \emph{Math.~Z.} \textbf{311}, 85 (2025),  \href{https://doi.org/10.1007/s00209-025-03874-9}{doi:10.1007/s00209-025-03874-9}}

\thanks{\ccby\, CC BY 4.0. This work is licensed under a \href{https://creativecommons.org/licenses/by/4.0/deed.en}{Creative Commons Attribution 4.0 License}}

\begin{abstract} We consider holomorphic foliations by curves on compact complex manifolds, for which we investigate the existence of foliated projective structures (projective structures along the leaves varying holomorphically)CY Cergy Paris Université  that satisfy particular uniformizability properties. Our results show that the singularities of the foliation impose severe restrictions for the existence of such structures. A foliated projective structure separates the  singularities of a foliation into parabolic and non-parabolic ones. For a strongly uniformizable foliated projective structure on a compact K\"ahler manifold, the existence of a single non-degenerate, non-parabolic singularity implies that the foliation is completely integrable. We establish an index theorem that imposes strong cohomological restrictions on the foliations having only non-degenerate singularities that support foliated projective structures making all of them  parabolic. As an application of our results, we prove that, on a projective space of arbitrary dimension,  a foliation by curves  of degree at least two, with only non-degenerate singularities, does not admit a strongly uniformizable foliated projective structure.\end{abstract}

\maketitle

\section{Introduction}
For a complex curve, a projective structure is an atlas for its complex structure with charts taking values in the complex projective line \(\PP^1\), and changes of coordinates in \(\mathrm{Aut}(\PP^1)\). Such structures may be considered along  the leaves of a one-dimensional foliation. Given a singular holomorphic foliation by curves \(\mathcal{F}\) on a complex manifold of dimension \(n+1\), with singular set of codimension at least two, a  \emph{foliated projective structure} is an atlas of holomorphic charts  \( (z_i, t_i)\), defined in the complement of the singular points of \(\mathcal{F}\), taking values in \(\PP^1\times\CC^{n}\), in which \(\mathcal{F}\) is given by the level sets of \(t_i\), with changes of coordinates  of the form 
\[ (z_j, t_j) = \left(\frac{a(t_i) z_i+ b(t_i)}{c(t_i)z_i+d(t_i)}, t_j(t_i)\right).\]
Through such an object, every leaf of \(\mathcal{F}\) inherits a  projective structure, which varies holomorphically  from leaf to leaf 
(see \cite[Section~2.2]{DG} for equivalent definitions).

The existence of these foliated projective structures was the subject of our previous article \cite{DG}. In the present work, we study foliated projective structures from a qualitative point of view, with interest in problems   related to the holomorphic  simultaneous uniformization of the leaves of holomorphic foliations.

A projective structure on a curve is said to be \emph{uniformizable} if it is the quotient of an open subset of \(\PP^1\) under the action of a subgroup of \(\mathrm{Aut}(\PP^1)\), and \emph{strongly uniformizable} if it is uniformizable and the uniformizing open set is simply connected. For example, the  projective structure induced by the \emph{Poincar\'e uniformization}, the uniformization that follows from the Koebe-Poincar\'e uniformization theorem, is strongly uniformizable, and  so are the projective structures induced by quasi-Fuchsian groups. Projective structures given by Schottky uniformizations are uniformizable, although not strongly so. A foliated projective structure is said to be \emph{uniformizable} or \emph{strongly uniformizable} if such is the projective structure that it induces on each and every one of its leaves. 

The aim of this article is to investigate the restrictions for the uniformizability and strong uniformizability  of foliated projective structures imposed, individually and collectively, by the singular points of a foliation.  

The problem of the simultaneous uniformization of the leaves of holomorphic foliations by curves has deserved a lot of attention.  Let us mention two directions (by no means the only ones) along which it has been addressed. We have Ilyahskenko's notion of \emph{simultaneous uniformization}, in which uniformizations vary holomorphically from leaf to leaf, but which are defined in some \emph{covering tubes}, and not on the leaves themselves (such a uniformization may give a leaf various inequivalent projective structures). We refer the reader to Ilyashenko's survey \cite{Ilyashenko-handbook}, as well as to Glutsyuk's article \cite{Glutsyuk-simultaneous} for precise definitions and statements, and limit to say that, for a  foliation whose leaves are hyperbolic Riemann surfaces, admitting a strongly uniformizable foliated projective structure is a stronger condition than admitting a  simultaneous uniformization in the sense of Ilyashenko. Another direction concerns the study of the variation of the Poincar\'e uniformization of the leaves of a foliation, which developed around the seminal works of Verjovsky \cite{verjovsky} and Candel \cite{candel}. For instance, within this trend of ideas, Lins Neto studied the Poincar\'e uniformization of the leaves of singular holomorphic foliations by curves. He proved that for a foliation by curves  on \(\PP^{n}\), of degree \(d\geq 2\), having only non-degenerate singularities, all of its leaves  are uniformized by the Poincar\'e disk, and their uniformizations vary continuously from leaf to leaf \cite{lins-simultaneous}  (see also \cite{candel-gomezmont}). In particular, the leaves of these foliations are endowed with projective structures, which vary continuously from leaf to leaf.

As an application of our results, we will prove the following theorem:

\begin{thm}\label{thm:pn_nounif} Let \(\mathcal{F}\) be a foliation by curves on \(\PP^n\), of  degree \(d\geq 2\), all of whose  singularities are non-degenerate. Then \(\mathcal{F}\) does not admit a strongly uniformizable foliated projective structure.
\end{thm}
(This result should be weighed against the abundance of foliated projective structures supported by foliations by curves on \(\PP^n\): a foliation of degree \(d\geq 1\) has an affine space of dimension \(\binom{n+2d-2}{n}\) of such structures; see Section~\ref{sec:proj_space}.)

A foliated projective structure is equivalent to the data of a \emph{(foliated, projective) Christoffel symbol}, a map \(\Xi:T_\mathcal{F}\to\mathcal{O}(M)\), that to a vector field \(Z\) tangent to \(\mathcal{F}\) associates the holomorphic function \(\Xi(Z)\), in a way satisfying the modified Leibniz rule
\begin{equation}\label{eq:leibniz}\Xi(fZ)=f^2\Xi(Z)+fZ^2f-\textstyle\frac{1}{2}(Zf)^2;\end{equation}
in this description, projective coordinates along the leaves are given by the solutions of the differential equations given by the vector fields \(Z\) tangent to \(\mathcal{F}\) for which \(\Xi(Z)\equiv 0\) (see \cite[Section~2.2]{DG}). In the presence of a foliated projective structure, a singularity \(p\) of \(\mathcal{F}\) is said to be \emph{parabolic} if for every vector field \(Z\) defining \(\mathcal{F}\) in a neighborhood of \(p\), with singular set of codimension at least two, \(\Xi(Z)|_p=0\). By (\ref{eq:leibniz}), this will happen for every such \(Z\) if and only if it happens for a single one of them. (Remark~\ref{rmk:par} will give a rationale for this terminology.) At a singular point, a foliated projective structure induced by a vector field with singular set of codimension at least two is  parabolic.

Non-parabolic singularities of uniformizable foliated projective structures admit very simple normal forms:

\begin{thm}\label{thm:proj_lin} Let \(\mathcal{F}\) be a foliation endowed with a uniformizable foliated projective structure, \(p\) a singularity of \(\mathcal{F}\) that is non-parabolic with respect to the foliated projective structure. Then, there exist local coordinates centered at \(p\) in which \(\mathcal{F}\) is generated by a linear diagonal vector field  with integral eigenvalues whose projective Christoffel symbol is the  constant~\(-1/2\). \end{thm}

In particular, if \(p\) is a non-degenerate singularity of \(\mathcal{F}\) which is non-parabolic for the foliated projective structure, its eigenvalues are commensurable.  

In the strongly uniformizable case, a single non-degenerate, non-parabolic singularity imposes, in the case where the ambient manifold is K\"ahler, the ``meromorphic complete integrability'' of the foliation:

\begin{thm}\label{thm:rat1stint} Let \(\mathcal{F}\) be a foliation by curves on a compact complex K\"ahler manifold \(M\) of dimension \(n+1\) endowed with a strongly uniformizable foliated projective structure. If there exists a non-degenerate singularity of \(\mathcal{F}\) which is non-parabolic for the foliated projective structure, there exist an analytic space \(V\)	of dimension  \(n\) and a dominant meromorphic map \(f:M\dashrightarrow V\) such that the leaves of \(\mathcal{F}\) are contained in the fibers of~\(f\). 
\end{thm}

In the case of surfaces, the K\"ahler hypothesis  may be dismissed (Remark~\ref{rmk:fi_surfaces}). For a foliation with non-degenerate singularities  and a foliated projective structure which makes them all parabolic, we have the following index theorem. It is an extension of Theorem 5.1 in \cite{DG},  in which uniformizability plays no part.

\begin{thm}\label{thm:index_parab} Let \(M\) be a compact complex manifold  of dimension \(n\), \(\mathcal{F}\) a holomorphic foliation by curves, all of whose singular points are  non-degenerate, supporting a foliated projective structure for which they are all parabolic. Let \(\varphi(x_1,\ldots,x_{n+1})\) be a symmetric homogeneous polynomial of degree \(n+1\), and, for \(i=0,\ldots,n+1\), let \(\widehat{\varphi}_i(x_1,\ldots,x_n)\) be the symmetric homogeneous polynomial of degree \(i\) defined by
\[\varphi(x_1,\ldots,x_n,x_{n+1})=\sum_{i=0}^{n+1}x_{n+1}^{n+1-i}\widehat{\varphi}_i(x_1,\ldots,x_n).\]
Then, 
\begin{equation}\label{eq:index_qf}\sum_{k>0}  c_1^{2k}(T_\mathcal{F})\widehat{\varphi}_{n-{2k}}(c(TM-T_\mathcal{F}))=0.\end{equation}\end{thm}
Notice that the last sum does not include the term for \(k=0\). In this statement, as in the Baum-Bott index theorem (which  will be used for its proof), evaluating a symmetric polynomial on the Chern polynomial of a virtual bundle consists in evaluating it at the associated Chern roots \cite{baum-bott}. 

Our results imply that manifolds of even dimensions admitting strongly uniformizable foliated projective structures are very special:  
\begin{corollary}\label{coro:unif-typegeneral} Let \(M\) be a compact complex algebraic manifold of even dimension, \(\mathcal{F}\) a holomorphic foliation by curves on \(M\) with only non-degenerate, non-dicritical singularities. If \(\mathcal{F}\) supports a strongly uniformizable foliated projective structure, \(\mathcal{F}\) is not of general type.
\end{corollary}

Foliations which are not of general type have been classified in dimension two (within the \emph{birational classification of foliations on surfaces} \cite{brunella-birational}), and, at least in  principle, can be described in higher dimensions as well. The result does not hold for odd-dimensional manifolds: as we shall see in Proposition~\ref{prop:oddgentype}, there  exist regular foliations of general type on  odd-dimensional manifolds which support strongly uniformizable foliated projective structures.

Despite the above negative results, we do have interesting  examples of strongly uniformizable foliated projective structures on foliations with non-trivial singular set. Two families will be described in Section~\ref{sec:examples}.  

An important ingredient in the proof of Theorem~\ref{thm:rat1stint} is a description, due to Jorge Pereira, of the set formed by the  compact analytic subspaces invariant by a foliation on a K\"ahler compact manifold. We will discuss this result, along with some of its consequences, in Appendix~\ref{app:cp}.

\section{Curves, foliations,  projective structures and uniformizability}

\subsection{Curves}\label{ssec:curves} A projective structure on a (connected) curve \(C\) comes with a \emph{monodromy} representation  \(\mathrm{mon}:\pi_1(C)\to \mathrm{PSL}(2,\CC)\) and a \emph{developing} map \(D:\widetilde{C}\to \PP^1\), an equivariant local biholomorphism.  A projective structure on a curve \(C\) is said to be \emph{uniformizable} if  it is isomorphic to the quotient of an open subset of \(\PP^1\) under the action of a subgroup of \(\mathrm{PSL}(2,\CC)\), and \emph{strongly uniformizable} if, moreover, this open subset is simply connected. A projective structure on a curve is strongly uniformizable if and only if its developing map is injective.

For a projective structure \(\sigma\) on a connected curve \(C\), and \(C'\subseteq C\) a connected open subset, for the induced projective structure \(\sigma'\) on \(C'\),
\begin{itemize}\item  if \(\sigma\) is uniformizable, so is \(\sigma'\);
\item if \(\sigma\) is strongly uniformizable, and the inclusion \(C'\hookrightarrow C\) injects the fundamental group of \(C'\), \(\sigma'\) is strongly uniformizable as well.
\end{itemize}

Projective geometry is the most general of the geometries given by transitive and faithful holomorphic actions of complex Lie groups on curves, the others being affine and translation geometry (given by the action of the affine group, or of its translation subgroup, on \(\CC\)). The uniformizability property can be considered for all of these geometries, and a certain framework is common to them all.  By a straightforward adaptation of \cite[Prop.~1.6.1]{guillot-survey}, we have:
\begin{prop}\label{prop:sc_develop} For a projective structure on a curve \(C\), the following are equivalent:
\begin{itemize}
\item the projective structure is uniformizable;
\item for the covering map 
\(\widehat{C}\to C\) associated to the kernel of the monodromy, the induced developing map \(\widehat{D}:\widehat{C}\to\PP^1\) is one-to-one;
\item for every open path \(\alpha:[0,1]\to C\), and for every lift \(\widetilde{\alpha}:[0,1]\to \widetilde{C}\), the path \(D\circ\widetilde{\alpha}\) is  also an open one.
\end{itemize}
\end{prop}

Some obstructions for uniformizability  are localized at the singularities of projective structures. Let us discuss them based on \cite[Section~3.2]{DG}. Let \(\Delta\subset\CC\) be the unit disk, \(\Delta^*=\Delta\setminus\{0\}\). For a projective structure on \(\Delta^*\), the origin is said to be a \emph{Fuchsian singular point} if \(z^2\Xi(\indel{z})\) extends as a holomorphic function to~\(0\). In such a case, if 
\begin{equation} \label{eq:def:ram_ind} z^2\Xi\left(\del{z}\right)= \frac{1}{2}\left(1-\frac{1}{\nu^2}\right)+\cdots,
\end{equation}
\(\nu\in \CC^*\cup\{\infty\}\) is said to be the \emph{(projective) ramification index} of the singular point. It is only well-defined up to sign, and is independent of the chosen coordinate. In this case, from (\ref{eq:def:ram_ind}), by (\ref{eq:leibniz}),
\begin{equation} \label{eq:ram_ind}\Xi\left((\lambda z+\cdots)\del{z}\right)=-\frac{\lambda^2}{2\nu^2}+\cdots.\end{equation}
Reciprocally, if \(\Xi((\lambda z+\cdots)\indel{z})\) is holomorphic, the projective structure has a Fuchsian singular point, and the ramification index is given by the above expression.

\begin{prop}  For a uniformizable projective structure on \(\Delta^*\), the origin is a Fuchsian singular point, and the ramification index is either an integer or infinity.
\end{prop}
The result is a projective analogue of Proposition~6 in \cite{guillot-rebelo}, to which we refer for a complete proof (see also \cite[Section~3.2.1]{DG}).

\begin{proof}[Sketch of proof] Let \(D:\Delta^*\to \PP^1\) be the (possibly multivalued) developing map of the projective structure.

If the monodromy is finite, of order \(\nu\), we may suppose that it is generated by \(w\mapsto \rho w\), for \(\rho\) a primitive \(\nu\)-th root of unity. Since the projective structure is uniformizable, \(D^\nu:\Delta^*\to \PP^1\) is a well-defined, one-to-one function. It extends holomorphically to \(0\), and may be supposed to map \(0\) to \(0\), giving a coordinate \(z\) on \(\Delta\) in which the developing map is given by \(z\mapsto z^{1/\nu}\). The projective structure is induced by the multivalued vector field \(\nu z^{1-1/\nu}\indel{z}\),   pull-back of \(\indel{w}\) under the developing map.  From \(\Xi(\nu z^{1-1/\nu}\indel{z})\equiv 0\) and (\ref{eq:leibniz}),
\(\Xi (\indel{z})=\frac{1}{2}(1-1/\nu^2)z^{-2}\): the singularity is Fuchsian, and has ramification index~\(\pm\nu\).
	
If the monodromy is parabolic, we may suppose that it is generated by \(w\mapsto w+2\ii\pi\). By the uniformizability hypothesis, \(\ee^D:\Delta^*\to \CC\setminus\{0\}\) is one-to-one, and thus extends holomorphically to \(0\), mapping \(0\) to \(0\). This gives a coordinate \(z\) in \(\Delta\) for which the  developing map is given by \(\log(z)\). The pull-back of the vector field \(\indel{w}\) is \(z\,\indel{z}\), which induces the projective structure. Since \(\Xi(z\,\indel{z})\equiv 0\), \(\Xi(\indel{z})=\frac{1}{2}z^{-2}\): the singularity is Fuchsian and has infinite ramification index.

The other potential monodromies are incompatible with the uniformizability condition. \end{proof}

\begin{lemma}\label{lemma:completecurve}  Let \(C\) be a \ (not necessarily compact) connected curve,  \(p\in C\), and let \(C\setminus\{p\}\) be endowed with a strongly uniformizable projective structure. Then, either 
\begin{itemize}
\item the projective structure extends to \(p\), and \(C\) is a compact rational curve, with its canonical projective structure (in particular, \(p\) is a Fuchsian singularity with ramification index \(\pm 1\)); or 
\item  \(p\) is a Fuchsian singularity with infinite ramification index. 
\end{itemize}
\end{lemma}

\begin{proof} Let \(\Delta\) be a small disk around \(p\), \(i:(\Delta,p)\to (C,p)\) the inclusion. The map \(i_*:\pi_1(\Delta\setminus\{p\})\to \pi_1(C\setminus\{p\})\) is either trivial or injective, since,  for every closed path in \(C\setminus\{p\}\), we have that either it is homotopically trivial, or its class in the fundamental group has infinite order. In the first case, \(C\setminus\{p\}\) is simply connected, and \(C\) is a compact rational curve; since the developing map is defined on \(\Delta\setminus\{p\}\), the point \(p\) has ramification index \(\pm 1\): the singularity is an apparent one, and the projective structure extends to a neighborhood of \(p\).  In the second case, \(\Delta\setminus\{p\}\) injects its fundamental group into that of \(C\setminus\{p\}\), and the projective structure induced on \(\Delta\setminus\{p\}\) is strongly uniformizable; this implies that its monodromy is infinite, and, by the previous proposition, that it is Fuchsian with infinite ramification index. \end{proof}

The following proposition sheds light on the nature of the singularities that appear in the second item of this lemma (it will not be used further on).

\begin{prop} For a projective structure on \(\Delta^*\), the following are equivalent:
\begin{enumerate}
\item \label{sing:parab-fuch} \(0\) is a Fuchsian singular point with infinite ramification index;
\item \label{sing:parab-vf} the projective structure is induced by a vector field vanishing at \(0\), having non-zero linear part; and
\item \label{sing:parab-unif} a neighborhood of \(0\) in \(\Delta^*\) is isomorphic, for some \(N\), to the quotient of  \(\Omega_N=\{w\in\CC\mid \Re(w)<N\}\) under the action of \(w\mapsto w+2\ii\pi\), and, in particular, the projective structure is uniformizable.
\end{enumerate} 
\end{prop}
\begin{proof} \underline{(\ref{sing:parab-fuch})\(\Rightarrow\)(\ref{sing:parab-vf})}. If \(0\) is a Fuchsian singular point with infinite ramification index, from (\ref{eq:ram_ind}), \(\Xi(z\,\indel{z})\) vanishes at \(0\). By (\ref{eq:leibniz}), for a function \(f\), the condition \(\Xi(fz\,\indel{z})= 0\) is equivalent, for \(h=f'/f\), to \(zh'+h+\frac{1}{2}zh^2+z^{-1}\Xi(z\,\indel{z})=0\). By the theorem of Briot and Bouquet \cite[\S~12.6]{ince}, this equation has a solution \(h\) which is holomorphic at \(0\), and there is thus a non-vanishing \(f\) such that \(\Xi(fz\,\indel{z})=0\). \underline{(\ref{sing:parab-vf})\(\Rightarrow\)(\ref{sing:parab-unif})}. Let the projective structure be given by the vector field \(fz\,\indel{z}\); up to multiplying by a constant, suppose that \(f(0)=1\). A developing map \(D\) for the projective structure (inverse of the solution of the differential equation given by the vector field) is given by \(z\mapsto \int^z  \dd s/(sf(s))\); if \(1/f(s)=1+sh(s)\),  \(D(z)=\log(z)+H(z)\), for a holomorphic function \(H\) with \(H'=h\), and \(\ee^D\) is one-to one. Its inverse gives the uniformizing map. \underline{(\ref{sing:parab-unif})\(\Rightarrow\)(\ref{sing:parab-vf})}. The map \(\ee^w:\Omega_N\to\Delta^*\) realizes the quotient by \(w\mapsto w+2\ii\pi\), and the image of the invariant vector field that induces the projective structure, \(\indel{w}\), is \(z\,\indel{z}\). \underline{(\ref{sing:parab-vf})\(\Rightarrow\)(\ref{sing:parab-fuch})}. It follows directly from equation~(\ref{eq:leibniz}). 
\end{proof}

\begin{rmk}\label{rmk:par}
If \(p\) is a non-degenerate singular point  of a foliation \(\mathcal{F}\), that is parabolic with respect to a foliated projective structure, and \(\gamma\) is a germ of curve  through \(p\), invariant by \(\mathcal{F}\), the restriction of the projective structure to \(\gamma\) is described by the previous proposition: it  is uniformizable, and has parabolic monodromy. 
\end{rmk}

\subsection{Foliations}  We will use leisurely  standard notions related to foliations by curves, like the \emph{tangent bundle} \(T_\mathcal{F}\) of a foliation \(\mathcal{F}\), or its dual, the \emph{canonical bundle} \(K_\mathcal{F}\); we refer the reader to the first chapters of \cite{brunella-birational} for a detailed presentation.  

A (singular, holomorphic) foliation by curves on a complex manifold \(M\) is given by a collection of pairs \(\{(U_i,Z_i)\}\), with \(U_i\subset M\)  an open subset, \(M=\cup_i U_i\), and \(Z_i\) a holomorphic vector field on \(U_i\), with singular set of codimension at least two, such that there is a non-vanishing function \(g_{ij}\) on \(U_i\cap U_j\), for which \(Z_i=g_{ij}Z_i\). The \emph{singular set} of \(\mathcal{F}\) is, on \(U_i\), given by the locus of zeros  of \(Z_i\). A singular point \(p\) of \(\mathcal{F}\) is \emph{non-degenerate} if \(\mathcal{F}\) is generated, on a neighborhood of \(p\), by a vector field vanishing at \(p\) having a  linear part   with no vanishing eigenvalues. In particular, non-degenerate singularities are isolated.

For a compact curve \(C\) invariant by a one-dimensional foliation \(\mathcal{F}\), we have the Poincar\'e-Hopf formula 
\begin{equation}\label{eq:ph_inv}
\mu(\mathcal{F}, C )=K_\mathcal{F}\cdot C+2-2g(C),
\end{equation}
in which \(g(C)\) is the geometric genus of \(C\), and \(\mu(\mathcal{F}, C )\) is defined as a sum of local contributions as follows: for each parametrized branch of \(C\) trough a point \(p\), and a vector field \(Z\) defining \(\mathcal{F}\) on a neighborhood of \(p\), the local contribution to \(\mu(\mathcal{F}, C )\) is the number of zeroes of the restriction of \(Z\) to the branch of \(C\) (see \cite[Prop.~2.2]{CCGDLF}; see also \cite[Prop.~2.8]{linsneto-soares} and \cite[Prop.~2.3]{brunella-birational}).

Foliated projective structures were the subject of our previous work \cite{DG}, to which we will often refer. If a foliation on a manifold is endowed with a foliated projective structure, any open subset inherits both a foliation and a foliated projective structure, which will be uniformizable if so is the original one. Uniformizability is thus  susceptible of being studied  locally (strong uniformizability is, in part, a global property).

Foliated translation structures are induced by vector fields with singular sets of codimension two, and their uniformizability is equivalent to the single-valuedness of the solutions of the vector field; an approach for the study of this property within the theory of holomorphic foliations has developed around the early work of Rebelo \cite{rebelo-singularites} (see \cite{guillot-survey} for a panoramic presentation). This has also motivated the study of the uniformizability of some foliated affine structures (see \cite[Section~4]{guillot-rebelo}, and the  discussion in \cite[Section~1.6]{guillot-survey}).

Some facts established for foliated affine and translation structures extend naturally to the projective case.

\begin{prop}\label{prop:unif_closed} Let \(\mathcal{F}\) be a foliation by curves on the manifold \(M\)  endowed with a foliated projective structure. The union of the leaves where the structure is not uniformizable is an open  subset of~\(M\).
\end{prop}
\begin{proof} Suppose that \(M\) is of dimension \(n+1\). Let \(L\) be a leaf of \(\mathcal{F}\) in restriction to which the projective structure is not uniformizable. Let \(\alpha_0:[0,1]\to L\) be an open path that develops onto a closed one. The standard construction leading to the definition of the holonomy of a leaf of a foliation can be adapted to the foliated projective structure, to ensure the following: there exist a simply connected curve \(E\), a path \(\widetilde{\alpha}:[0,1]\to E\), an \(n\)-dimensional ball \(\BB^{n}\), and immersions \(D:E\to\PP^1\) and  \(i:E\times \BB^{n}\to M\) such that
\begin{enumerate}	
\item 	for every \(t\in \BB^{n}\), \(i|_{E\times\{t\}}\) is tangent to \(\mathcal{F}\);  \item the foliated projective structures on \(E\times \BB^{n}\)  induced, on the one hand,  by~\(i\), and, on the other, by \((z,t)\mapsto D(z)\), agree; \item \(\alpha(s)=i(\widetilde{\alpha}(s),0)\). 
\end{enumerate}
A development of  \(\alpha_t(s)=i(\widetilde{\alpha}(s),t)\) is  given by \(D(\widetilde{\alpha}(s))\). For \(t\) close enough to \(0\), \(\alpha_t\) is an open path developing into a closed one, and, thus, for the corresponding leaf of \(\mathcal{F}\), the projective structure is not uniformizable.
\end{proof}

\begin{prop}\label{prop:holmon} Let \(\mathcal{F}\) be a foliation by curves endowed with a uniformizable foliated projective structure. Let \(L\) be a leaf of \(\mathcal{F}\), and let \(\mathrm{hol}:\pi_1(L,p)\to \mathrm{Diff}(\CC^{n},0)\) and \(\mathrm{mon}:\pi_1(L,p)\to\mathrm{PSL}(2,\CC)\) be, respectively, the holonomy and monodromy representations. Then, \(\ker(\mathrm{mon})\subseteq\ker(\mathrm{hol})\).  \end{prop}
\begin{proof}
Let \(\alpha_0:[0,1]\to L\) be a closed path of trivial monodromy, and consider, with it, the objects in the proof of Proposition~\ref{prop:unif_closed}. Consider the transversal \(T=i(\widetilde{\alpha}(0)\times \BB^{n})\). The path \(\alpha_0\) develops onto a closed path, and so do the paths \(\alpha_t\), and they must in consequence be closed, starting and ending at \(T\). This is exactly the triviality of the holonomy along~\(\alpha_0\). \end{proof}

Analogues of Proposition~\ref{prop:holmon} for foliated affine and translation structures appear in \cite[Prop.~2.7]{rebelo-singularites} and \cite[Fund. Lemma]{guillot-rebelo}. An affine version of Proposition~\ref{prop:unif_closed} can be found in \cite[Cor.~12]{guillot-rebelo}.

\section{Examples}\label{sec:examples}

Some regular foliations admit strongly uniformizable foliated projective structures. On compact surfaces uniformized by the bidisk, the horizontal and vertical foliations  are canonically endowed with them. The Inoue surfaces \(S_M\) have two foliations (one by disks, the other by entire curves), both of which have tautological strongly uniformizable foliated projective structures. The foliated projective structures obtained by the suspension of a strongly uniformizable projective structure on a curve is strongly uniformizable as well \cite[Ex.~2.16]{DG}. Global holomorphic vector fields with nontrivial singular set of codimension at least two on compact manifolds give foliations with singularities admitting strongly uniformizable foliated translation (hence projective) structures. More generally, from Brunella's works on the uniformization of foliations by curves, we have that a foliation by curves  on a compact K\"ahler manifold having only parabolic leaves (with Brunella's ad hoc definition of leaf), has a uniformizable foliated affine structure, with charts given by the ``universal covering tubes'' associated to the foliation; see \cite[Section~3.3]{brunella-ps}.

We are not aware of many examples of strongly uniformizable foliated projective structures that have nontrivial singular set and that are truly projective (that are neither given by global holomorphic vector fields nor reduce to foliated affine structures). We hereby present two families of such foliated projective structures.

\subsection{Hilbert modular foliations}

Let \(\HH=\{z\in \CC\mid \Im(z)>0\}\). The group \(\mathrm{PSL}(2,\RR)\times \mathrm{PSL}(2,\RR)\) acts diagonally on \(\HH^2\) by preserving both the horizontal and vertical foliations by disks, together with their canonical foliated projective structures. For a torsion-free, discrete subgroup \(\Gamma\subset \mathrm{PSL}(2,\RR)\times \mathrm{PSL}(2,\RR)\), the quotient manifold inherits two foliations   with a strongly uniformizable foliated projective structure each.

For a square-free integer \(d\), and \(\mathfrak{o}\) the ring of algebraic integers of \(\QQ(\sqrt{d})\), the \emph{Hilbert modular group} \(\mathrm{PSL}(\mathfrak{o})\) embeds as a non-uniform, irreducible lattice in \(\mathrm{PSL}(2,\RR)\times \mathrm{PSL}(2,\RR)\). The quotient \(\mathrm{PSL}(\mathfrak{o})\backslash\HH^2\), the \emph{Hilbert modular surface}, is an analytic space with finitely many singular points, corresponding to points in \(\HH^2\) with non-trivial stabilizer, and a finite number of ends (cusps). The one-point compactification of a cusp is a normal analytic space; in its resolution, the exceptional divisor is a cycle of smooth rational curves intersecting transversely, in which the self-intersections of the curves can be explicitly described \cite{hirzebruch-hilbert}. The vertical and horizontal foliations on \(\mathrm{PSL}(\mathfrak{o})\backslash\HH^2\), the \emph{Hilbert modular foliations}, extend to this compactification while leaving these rational curves invariant \cite[Ex.~9.4]{brunella-birational}.

Let us show that the foliated projective structure extends to the rational curves in the exceptional divisor of the desingularization of a cusp. On a neighborhood of a point in the intersection of two of the rational curves of the cycle, there are coordinates \(u\) and \(v\) (locally defining the two rational curves) and non-zero constants \(w\) and \(w'\), \(w\neq w'\), such that \(u\) and \(v\) are related to the global coordinates \((z_1,z_2)\) of \(\HH^2\) by
\[
\begin{split}
	2\ii\pi z_1 & =  w\log u+\log v, \\
	2\ii\pi z_2 & =  w'\log u+\log v
\end{split}
\]
(this is formula (9) in \cite[Section 2.3]{hirzebruch-hilbert}). The coordinate vector field \(\indel{z_1}\) induces one of the foliated projective structures, and, in the above coordinates, it  reads
\[\frac{2\ii\pi}{w-w'}\left(u\del{u}-w'v\del{v}\right).\]
Thus, the foliated projective structure induced by \(\indel{z_1}\) extends to the rational curves of the cycles, and, on a neighborhood of the intersection of two of them, is  induced by a non-degenerate vector field (a parabolic singularity of the foliated projective structure). Each rational curve of the cycle has two singular points, and the projective structure on their complement is the natural one on~\(\CC/2\ii\pi\ZZ\).

By considering a finite-index, torsion-free subgroup of \(\mathrm{PSL}(\mathfrak{o})\) we obtain a compact complex surface with a strongly uniformizable foliated projective structure with non-degenerate, parabolic  singularities.

\subsection{The complex geodesic foliation of a complex-hyperbolic manifold} \label{sec:geod-fol}
Let \(\Gamma\subset\mathrm{PU}(n,1)\) be a torsion-free  lattice. Its projective action on \(\PP^{n}\) preserves the complex-hyperbolic ball \(\BB^n\subset \PP^n\), on which it acts freely and properly discontinuously, having as quotient a complex-hyperbolic manifold \(M=\Gamma\backslash \BB^n\). The intersection of a line  of \(\PP^n\) with  \(\BB^n\)  is  a totally geodesic subspace of \(\BB^n\), a \emph{complex geodesic}. Let \(\Omega\) be the set of couples \((p,l)\), where \(p\in\BB^n\) and \(l\) is a line in \(\PP^n\) passing through \(p\). Consider the non-singular foliation by curves \(\mathcal{F}\) on \(\Omega\) whose leaves are given by the couples in \(\Omega\) having a common line. The leaf through \((p,l)\) identifies naturally to the disk \(l\cap\BB^n\), and \(\mathcal{F}\) has thus a natural foliated projective structure. The group \(\Gamma\) acts freely and properly discontinuously on \(\Omega\), preserving \(\mathcal{F}\) together with its foliated projective structure, producing  a manifold endowed with both a non-singular  foliation by curves \(\mathcal{F}\), the \emph{complex geodesic foliation}, and a strongly uniformizable foliated projective structure along it. The manifold \(\Gamma\backslash\Omega\) is naturally biholomorphic to \(\PP(TM)\), the projectivization of the tangent bundle of~\(M\).

In the case where \(\Gamma\) is not uniform, Mok has shown that \(M\) can be compactified via a \emph{toroidal Mumford compactification} into a smooth manifold \(\overline{M}\) by gluing an abelian variety of  dimension \(n-1\) to each one of its finitely many ends. On its turn, this produces the compactification \(\PP(T\overline{M})\) of \(\PP(TM)\). We will review this construction  following \cite[Section~2.1]{mok}; we will show that the above foliation \(\mathcal{F}\) extends to \(\PP(T\overline{M})\) as a foliation with singularities \(\overline{\mathcal{F}}\), that the foliated projective structure on \(\mathcal{F}\) extends to \(\overline{\mathcal{F}}\), and that this extension is  strongly uniformizable.

Consider the Siegel domain presentation of complex-hyperbolic space
\[S_n=\{z\in\CC^n\mid \Im(z_n)>|z_1|^2+\cdots+|z_{n-1}|^2\}.\]
Write \(z\in \CC^n\) as \((z';z_n)\) with \(z'=(z_1,\ldots,z_{n-1})\).  Let \(W\subset \mathrm{PU}(n,1)\) be the group of transformations of \(S_n\) of the form
\begin{equation}\label{eq:action_mok}
(z';z_n)\mapsto\left(z'+a';z_n+2\ii\sum_{j=1}^{n-1}\overline{a_j}z_j+\ii\|a'\|^2+t\right),
\end{equation}
with \((a',t)\in \CC^{n-1}\times \RR\). Through its action on \(\PP^n\) via the embedding \((z_1,\ldots,z_n)\mapsto [z_1:\cdots:z_n:1]\), these transformations fix the point \(b=[0:\cdots:0:1:0]\), belonging to the boundary of \(S_n\) (\(W\) is the unipotent radical of the stabilizer of \(b\)). The first derived subgroup \(W'\) of \(W\) is given by the elements of the form \((0,t)\).

All local models for the ends of finite-volume complex-hyperbolic manifolds are given by quotients of neighborhoods of \(b\) by lattices  of \(W\) acting as above. Let \(\Gamma_0\) be such a lattice. Let \(\tau\in\RR\) be such that the intersection of \(\Gamma_0\) with \(W'\) is generated by \((0,\tau)\). Consider the quotient under the action of the latter, realized by
\begin{equation}\label{eq:mumfordmap}
(w_1,\ldots,w_n)=(z_1,\ldots,z_{n-1},\ee^{\rho z_n}),
\end{equation}
for \(\rho=2\ii\pi/\tau\). In the target space, let \(\Pi\) be the hyperplane \(w_n=0\).
After quotient, the action (\ref{eq:action_mok}) with \(t \in\tau\ZZ\) is given by
\[(w';w_n)\mapsto \left(w'+a';\exp\left(2\ii\rho\sum_{j=1}^{n-1}\overline{a_j}w_j+\ii\rho\|a'\|^2\right)w_n\right).\]
The restriction of the action to \(\Pi\) is given by translations, and the quotient of \(\Pi\) under \(\Gamma_0/\Gamma_0'\) gives the abelian variety compactifying this end.

Let \(\zeta_i=\dd z_i\) (as a function on \(TS_n\)), so that \(z_1,\ldots,z_n,\zeta_1,\ldots,\zeta_n\) are coordinates for \(TS_n\). On \(\PP(TS_n)\), both \(\mathcal{F}\) and the foliated projective structure on it are induced by the homogeneous vector field \(X=\sum \zeta_i\,\indel{z_i}\) on  \(TS_n\). Let \(\xi_i=\dd w_i\), so that \(w_1,\ldots,w_n,\xi_1,\ldots,\xi_n\) are coordinates on \(T\CC^n\).  The map (\ref{eq:mumfordmap}) extends to the tangent spaces. In the complement of \(\Pi\), \(X\) passes to the quotient as the homogeneous vector field on \(T\CC^n\) 
\[Y=\sum_{i=1}^n\xi_i\del{w_i}+\frac{\xi_n^2}{w_n}\del{\xi_n}.\]
In order to obtain a vector field on a chart of \(\PP(T\CC^n)\), restrict \(Y\) to a nonzero level set of its first integral \(\xi_j\) (\(j<n\)). Let \(Y_j\) denote this restriction. In the chosen chart, the induced foliation is given by the holomorphic vector field \(w_nY_j\), and extends to \(\Pi\) in a way tangent to it. The extended foliation has a singular set of codimension two, given by \(\Pi\cap\{\xi_n=0\}\).

From \(\Xi(Y_j)\equiv 0\), it follows that \(\Xi(w_n Y_j)=\frac{1}{2}\xi_n^2\):
the foliated projective structure extends to \(\Pi\). By Proposition~\ref{prop:unif_closed}, this structure is uniformizable. Let us show that it is strongly so. The vector field on \(\Pi\) given by the restriction of \(w_nY_j\)  is \(\xi_n^2\,\indel{\xi_n}\), and from \(\Xi|_\Pi(\xi_n^2\,\indel{\xi_n})=\frac{1}{2}\xi_n^2\), we have that \(\Xi|_\Pi(\xi_n\,\indel{\xi_n})\equiv 0\): the foliated projective structure is induced by the complete vector field \(\xi_n\,\indel{\xi_n}\). The leaves within \(\Pi\) are, in consequence, uniformized by
\[t\mapsto ((c_1,\ldots,c_{n-1},0),[v_1:\cdots:v_{n-1}:\ee^t]),\]
and the foliated projective structure is strongly uniformizable.

For cocompact lattices, these same foliations show that the conclusion of Proposition~6.3 in \cite{DG} does not hold for odd-dimensional manifolds, not even in the presence of strongly uniformizable foliated projective structures (cf. Corollary~\ref{coro:unif-typegeneral}).
\begin{prop}\label{prop:oddgentype} 
Let \(\Gamma\subset\mathrm{PU}(n,1)\) be a torsion-free cocompact lattice, and let \(M=\Gamma\backslash \BB^n\). The complex geodesic foliation \(\mathcal{F}\) on \(\PP(TM)\) is of general type (in fact, its canonical bundle \(K_\mathcal{F}\) is ample). 
\end{prop} 
\begin{proof} 
The tangent bundle of the foliation, \( T_{\mathcal F}\), identifies with the tautological line bundle \( O_{TM}(- 1)\) over \( \mathbb P (T M)\). Since the complex-hyperbolic Hermitian metric on \(\BB^n\) has negative curvature, the line bundle \(T_{\mathcal F}\) is  negative (see~\cite[\S 6]{kobayashi-ochiai-positive}). This implies that its dual \(K_\mathcal{F}\) is a positive line bundle, and hence, by the Kodaira embedding theorem, an ample one (see \cite[\S3.6]{kobayashi-vectorbundles} for a discussion around positivity and ampleness). 
\end{proof} 
 
\section{Local normal forms} \label{sec:normal_forms}

\subsection{Foliated projective structures} For the proof of Theorem~\ref{thm:proj_lin}, we will use the \emph{geodesic vector field} constructed  in \cite[Section~5.2]{DG}. We begin by recalling  its definition along with some of its properties.   Let \(Z\) be a vector field on a neighborhood \(U\) of \(0\) in \(\CC^n\), \(\mathcal{F}\) the foliation it generates. Consider a foliated projective structure along \(\mathcal{F}\), and let \(\rho:U\to\CC\) be the projective Christoffel symbol of \(Z\).  On \(U\times \CC^2\),  consider the \emph{projective geodesic} vector field, given, in the coordinates \((z_1,\ldots,z_n,\xi,\zeta)\), by
\begin{equation*} 
 X=\zeta Z+\zeta\xi\del{\zeta}+\left(\textstyle\frac{1}{2}\xi^2-\rho \zeta^2\right)\del{\xi}.\end{equation*}
Under the projection  \(\pi:U\times \CC^2\to U\), \(\zeta^{-1}X\) gets mapped  to \(Z\).  The natural parametrizations of the leaves of \(X\) give, when projected via \(\pi\), projective parametrizations of the leaves of \(\mathcal{F}\) since, for the foliated projective structure induced by \(X\), by (\ref{eq:leibniz}), \(\Xi(\zeta^{-1}X)=\rho\).

With the vector fields \(H=\zeta\,\indel{\zeta}+\xi\,\indel{\xi}\) and \(Y=2\,\indel{\xi}\), \(X\) satisfies the \(\mathfrak{sl}(2,\CC)\) relations 
\[  [H,X]=X,\; [H,Y]=-Y,\; [Y,X]=2H,\]
corresponding to the identification of \(H\), \(X\) and \(Y\) with the left-invariant vector fields on \(\mathrm{SL}(2,\CC)\) associated to the elements 
\[\left(\begin{array}{rr} \frac{1}{2} & 0  \\ 0  & -\frac{1}{2} \end{array} \right),\;\left(\begin{array}{rr} 0  & 1 \\ 0  & 0 \end{array} \right) \text { and } \left(\begin{array}{rr} 0 & 0 \\ -1 &  0\end{array} \right)\]
of its Lie algebra.  A  consequence of the above relations, and of the simple forms that take the integrations of \(H\) and \(Y\), is that, if \(\phi(t)=(z(t),\zeta(t),\xi(t))\) is a solution of \(X\), and \(A=\left(\begin{array}{cc}a & b \\ c & d \end{array}\right)\in \mathrm{SL}(2,\CC)\),
\begin{equation*}
\phi_A(t)=\left(z\left(\frac{at+b}{ct+d}\right),\frac{1}{(ct+d)^2}\zeta\left(\frac{at+b}{ct+d}\right),\frac{1}{(ct+d)^2}\xi\left(\frac{at+b}{ct+d}\right)-\frac{2c}{ct+d}\right)
\end{equation*} 
is again a solution of \(X\), one which projects onto the same leaf of \(\mathcal{F}\), and which depends only on the class of \(A\) in \(\mathrm{PSL}(2,\CC)\) (this can be established by a direct calculation; see also \cite[Prop.~2]{guillot-ihes}). This formula admits the following interpretation: when going from one integral curve of \(X\) to another one lying over the same leaf of \(\mathcal{F}\) by gliding along the fibers of \(\pi\), the translation parameter changes projectively, and the above formula gives the specific way in which this occurs. In particular, for a solution \(\phi(t)\) of \(X\) defined on a neighborhood of \(t=0\), upon transforming it by \(P=\left(\begin{array}{cc}a & 0 \\ c & d \end{array}\right)\), we obtain a solution with initial condition
\[\left(z(0),\frac{1}{d^2}\zeta(0),\frac{1}{d^2}\xi(0)-\frac{2c}{d}\right),\]
and, in this way, we obtain all the solutions of \(X\) with initial conditions in the same fiber of \(\pi\) as \(\phi\). The following are easily seen to be equivalent:
\begin{itemize}
\item the solutions \(\phi\) and \(\phi_P\) agree;
\item the transformation between the translation coordinates induced by \(\phi\) and \(\phi_P\) is a translation;
\item the class of \(P\) in \(\mathrm{PSL}(2,\CC)\) is trivial.
\end{itemize}

We have the following result (compare with \cite[Prop.~4]{guillot-ihes}):  

\begin{prop} In the above setting, if the foliated projective structure of \(\mathcal{F}\) is uniformizable, then so is the foliated translation structure induced by~\(X\) on the foliation it defines.
\end{prop}

The conclusion is equivalent to the fact that the vector field \(X\) is \emph{semicomplete}, that its solutions are single-valued in their maximal definition domain (see \cite{rebelo-singularites}, \cite[Prop.~1.2.1]{guillot-survey}).

\begin{proof} We will establish the result via the criterion presented in Proposition~\ref{prop:sc_develop}. Let \(L\) be a leaf of \(X\), and \(\alpha:[0,1]\to L\) a path that, with respect to the translation structure on \(L\), develops onto a closed one. We need to prove that \(\alpha\) is a closed path. The restriction \(\pi|_L:L\to \pi(L)\) is a local biholomorphism; it is a projective map with respect to the projective structure on \(L\) induced by the translation structure coming from \(X\). If \(\pi\circ\alpha\) is an open path in the leaf \(\pi(L)\) of \(\mathcal{F}\), it  develops, with respect to the projective structure of \(\pi(L)\),  onto a closed path (for a development of \(\alpha\) with respect to the translation structure on \(L\) is a development of \(\pi\circ\alpha\) with respect to the projective structure on \(\pi(L)\)), but this contradicts the uniformizability of the foliated projective structure on \(\mathcal{F}\). Let us now address the case in which \(\pi\circ\alpha\) is a closed path, while \(\alpha\) is an open one. Let \(z_0\) and \(z_1\) be local translation coordinates on \(L\) centered, respectively, at \(\alpha(0)\) and \(\alpha(1)\).  The projection \(\pi|_L\) identifies neighborhoods  of  \(\alpha(0)\) and \(\alpha(1)\) within \(L\) by a non-trivial fractional linear transformation relating \(z_0\) and \(z_1\), mapping \(z_0=0\) to \(z_1=0\). In particular, for values of \(t\) close to \(0\) but different from it, this fractional lineal transformation maps \(z_0=t\) to a point different from \(z_1=t\). If a small translation of \(\alpha\) within \(L\) is a path starting at \(z_0=t\), it ends at \(z_1=t\). Thus, a small translation of \(\alpha\) within \(L\) is a path that develops onto a closed one, and that projects to an open path on \(\pi(L)\), reducing this case to the previous one. We conclude that there is no open path in \(L\) that develops into a closed one: the foliated translation structure of \(X\) is uniformizable. 
\end{proof}

\begin{proof}[Proof of Theorem~\ref{thm:proj_lin}] The result is a generalization of Theorem B in \cite{guillot-fourier}, from which we will borrow some arguments. The setting is the one described in the beginning of this section; we keep the notations there set. The core hypothesis is that \(0\) is a non-parabolic singularity of the projective structure, that \(\rho(0)\neq 0\).

Let \(c\in\CC\), and consider the vector field \(H+c X\) on \(U\times\CC^2\). Denoting by \(\Phi_Z^t\) the flow of the vector field \(Z\) in time \(t\), we have that
\begin{equation}\label{for:conj_sc}\Phi_{H+cX}^t=\Phi^{-c}_{X}\Phi^t_{H}\Phi^c_{X}=\Phi^{-c}_{X}\Phi^{c\ee^t}_{X}\Phi^t_{H}=\Phi^{c(\ee^t-1)}_{X}\Phi^t_{H},
\end{equation}
as  follows from the integral form of the Lie bracket relation \([H,X]=X\) and the previously discussed identification of \(H\) and \(X\) with particular left-invariant vector fields on \(\mathrm{SL}(2,\CC)\). Formula (\ref{for:conj_sc}) shows that \(H+c X\) is semicomplete: its flow in time \(t\) it is unambiguously defined whenever it is so, precisely because \(X\) is semicomplete (this is how the uniformizability hypothesis on the foliated projective structure  comes into play). Moreover, and this is crucial, the same formula shows that the solutions of \(H+c X\) are, like those of \(H\), \(2\ii\pi\)-periodic. 

Let \(\xi_0\) be a square root of \(-2\rho(0)\), which is different from zero by hypothesis.  Let \(p=(0,\ldots,0,1,\xi_0)\in U\times\CC^2\); it projects to the singular point of \(\mathcal{F}\) in \(U\). Let \(F=\pi^{-1}(0)\) be the fiber of \(\pi\) through \(p\). Consider the vector field \(A=H-\xi_0^{-1}X\) on \(U\times\CC^2\). It vanishes at \(p\), preserves \(F\), and
\[\mathrm{D}A|_p = \left(\begin{array}{ccc|cc} & & & 0 & 0 \\ & -\frac{1}{\xi_0}\mathrm{D}Z|_0 & & \vdots & \vdots \\  & & & 0 & 0 \\ \hline  0 & \cdots & 0 & 0 & -1/\xi_0   \\ * & \cdots & * & -\xi_0 & 0 \end{array}\right).\]
The eigenvalues of the latter are those of \(-\xi_0^{-1}\mathrm{D} Z|_0\), plus \(1\) and \(-1\), the last two corresponding to directions tangent to \(F\), with the last one tangent to the direction of~\(X(p)\).  
 
We claim that \emph{the eigenvalues of the linear part of \(A\) at \(p\) are integers}, and that \emph{there exist coordinates \((x_1,\ldots,x_{n},u,v)\), centered at \(p\), in which \(X=\indel{v}\), \(F\) is given by \(\cap_{i=1}^n\{x_i=0\}\), and  
\begin{equation}\label{for:prepaff}
A=\sum_{i=1}^{n}\lambda_ix_i\del{x_i}+u\del{u}-v\del{v},\end{equation}
with \(\lambda_i\in \ZZ\)}. We will follow the proof of Lemma 2.5 in \cite{guillot-fourier} to some extent, and repeat some of its arguments here. The vector field \(X\) is tangent to \(F\). Choose coordinates \((x_1,\ldots,x_{n+1},v)\), centered at \(p\), in which \(F\) is given by \(\cap_{i=1}^n \{x_{i}=0\}\), and  \(X\) by \(\indel{v}\). In these, in a neighborhood of \(p\),  the leaf space of \(X\)  is realized by  \(\nu(x_1,\ldots,x_{n+1},v)= (x_1,\ldots,x_{n+1})\), with \(\nu(F)\)   given by \(\cap_{i=1}^n  \{x_{i}=0\}\). Since \([A,X]=X\), the vector field \(A\) preserves the foliation induced by \(X\), and induces a vector field \(\nu_*A\) on its leaf space. Since the germ of \(A\) around \(p\) is semicomplete and \(2\ii\pi\)-periodic, so is the germ of \(\nu_*A\) in a neighborhood of \(\nu(p)\). The flow of \(\Im(\nu_*A)\), the imaginary part of \(\nu_*A\), induces an analytic  action of the compact group \(\RR/2\pi\ZZ\)  by biholomorphisms. This action is holomorphically linearizable: by the Bochner-Cartan theorem \cite[Ch.~V, Thm.~1]{montgomery-zippin},  the map 
\begin{equation}\label{bochner-cartan}\Psi(x)=\frac{1}{2\pi}\int_0^{2\pi} (\mathrm D\Phi_{\nu_*A}^{\ii\theta}|_p)\Phi_{\nu_*A}^{-\ii\theta}(x)\,\dd\theta\end{equation}
is a local holomorphic change of coordinates, fixing the point \(p\), tangent to the identity at it, that  maps \(\Im(\nu_*A)\)  to its linear part at \(\nu(p)\). This change of coordinates maps \(\nu_*A\) to a holomorphic vector field whose imaginary part is a linear vector field with \(2\pi\)-periodic solutions: a diagonalizable linear vector field with integral eigenvalues, necessarily. Since \(A\) is tangent to \(F\),  the line \(\nu(F)\) is preserved by \(\Phi^t_{\nu_*A}\) (and hence also by \(\mathrm{D}\Phi^t_{\nu_*A}|_p\)) 
for every \(t\) for which \(\Phi^t_{\nu_*A}\) is defined. From this, from the original choice of coordinates, and from the explicit form of~(\ref{bochner-cartan}), we have that \(\Psi_*(\nu_*A)\) preserves the line \(\nu(F)\). In the leaf space of \(X\), let \(L\) be a linear change of coordinates  preserving \(\nu(F)\) such that \((L\circ\Psi)_*(\nu_*A)=\sum \lambda_i x_i\,\indel{x_i}\) with \(\lambda_i\in \ZZ\). Observe that we must have \(\lambda_{n+1}=1\), for the eigenvalue of \(\nu_*A\) tangent to \(\nu(F)\) is \(1\). By extending these changes of coordinates to a neighborhood of \(p\) by \((x,v)\mapsto (L\circ \Psi(x),v)\), we obtain  coordinates \((x_1,\ldots,x_{n+1},v)\), centered at \(p\), in which \(X\) retains the expression \(\indel{v}\), and in which
\[A=\sum_{i=1}^{n+1} \lambda_i x_i\del{x_i}+f(x,v)\del{v},\]
for some function \(f\) vanishing at \(0\). The relation \([A,X]=X\) implies that \(f(x,v)=h(x)-v\) for some function \(h\). Let us show that, up to  a new change of coordinates, we may suppose that \(h\equiv 0\). Let \(h=\sum a_Ix^I\) (we use multi-index notation). The solution of \(A\) with initial condition \((c,v_0)\in\CC^{n+1}\times \CC\) (close to \(0\)) is given by \(x_i(t)=c_i\ee^{\lambda_i t}\), for \(i=1,\ldots,n+1\), and by
\[
\begin{split}
v(t) & =  \ee^{-t}\left(v_0+ \int_0^t \ee^s h(c_1\ee^{\lambda_1 s},\ldots,c_{n+1}\ee^{\lambda_{n+1}s} )\,\dd s\right) \\
	& =  \ee^{-t}\left(v_0+\int_0^t  \sum   a_Ic^I \ee^{(\langle\lambda,I\rangle+1) s}\,\dd s\right).
\end{split}
\]
Since the flow of \(A\) is  \(2\ii\pi\)-periodic, we must have that
\[\sum_{ \langle\lambda,I\rangle=-1}   a_Ic^I= 0\]
for all \(c\) sufficiently close to \(0\), and thus \(a_I=0\) if \(\langle\lambda,I\rangle=-1\). Let \(\overline{v}=v+g(x)\) for some holomorphic function \(g\) vanishing at \(0\), and consider coordinates \((x_1,\ldots, x_{n+1},\overline{v})\). Since \(H \overline{v}=1\), \(H\) has the same expression in the new coordinates. Let us show that \(g\) can be chosen so that  \(A  \overline{v}=-\overline{v}\). If \(g=\sum b_Ix^I\), this condition 
is formally equivalent to \(a_I+(\langle\lambda,I\rangle+1)b_I=0\) for all \(I\), and the previous observation ensures that there is a formal solution  
\[b_I=\begin{cases}
	0, & \text{if \(\langle\lambda,I\rangle= -1\)},\\
	-{\displaystyle \frac{a_I}{1+\langle\lambda,I\rangle}}, & \text{if \(\langle\lambda,I\rangle\neq -1\)}, 
\end{cases}\]
which is easily seen to be convergent. After relabeling \(x_{n+1}\) as \(u\), \(A\) is given by~(\ref{for:prepaff}), and \(F\) and \(X\) have the sought form. This establishes the claim.

Let \(k=-1/\xi_0\),  \(y_i=x_i(v+k)^{\lambda_i}\), \(w=u(v+k)\). In the coordinates \((y_1,\ldots,y_{n},w,v)\),   
\[X=\frac{1}{v+k}\left(\sum_{i=1}^n\lambda_iy_i\del{y_i}+w\del{w}+(v+k)\del{v}\right),\]
and \(H=-(v+k)\indel{v}\). Consider the codimension-one manifold \(\Sigma\) given by \(\{w=0\}\). It is transverse to \(\pi\), and is saturated by both \(H\) and \(X\). Let us restrict to it.  The fibers of  \(\pi|_\Sigma\)  are the integral curves of \(H\), and, in a neighborhood of \(p\) within \(\Sigma\), \(\pi|_\Sigma\)   is given by \((y_1,\ldots,y_n,0,v)\mapsto (y_1,\ldots,y_n)\). Under this projection,  \(\mathcal{F}\), together with its foliated projective structure, is realized by the image of \(X|_\Sigma\) under \(\pi|_\Sigma\). (Notice that, since the foliated projective structure is induced by the quotient of \(X\)  by \(H\), and \([X,H]=H\),  by restricting \(X\) to \(\Sigma\) we have, by the way, reduced the foliated projective structure to an affine one.) For the foliated projective structure generated by \(X\), from \(\Xi(X)\equiv 0\), by (\ref{eq:leibniz}), we have that \(\Xi((v+k)X)\equiv -1/2\). Since \([H,(v+k)X]=0\), the vector field \((v+k)X|_\Sigma\) may be projected to the leaf space of \(H|_\Sigma\) as a vector field. This projection reads simply \(\sum_{i=1}^n\lambda_iy_i\,\indel{y_i}\). \end{proof} 

With the normalization of Theorem~\ref{thm:proj_lin}, the \(\lambda_i\) are the ramification indices  of the projective structures along the coordinate axis at \(p\), as discussed in  Section~\ref{ssec:curves}. Up to a simultaneous  change of sign, they are the \emph{principal projective ramification indices} of the foliated projective structure  at \(p\), as defined in \cite[Section~3.2]{DG}. 

For  strongly uniformizable foliated projective structures,  this  observation implies, with Lemma~\ref{lemma:completecurve}, that \(\lambda_i=\pm 1\) for all \(i\); Theorem~\ref{thm:proj_lin} specializes as follows:

\begin{corollary}\label{coro:norforstrong}  Let \(\mathcal{F}\) be a foliation on a manifold \(M\) endowed with a strongly uniformizable foliated projective structure. Let \(p\in M\) be a non-degenerate singularity of \(\mathcal{F}\) which is not parabolic. Then, there exist coordinates around \(p\) in which \(\mathcal{F}\) is generated by the vector field \(\sum_{i\leq k}z_i\,\indel{z_i}-\sum_{i>k}z_i\,\indel{z_i}\), whose Christoffel symbol is constant and equal to~\(-1/2\).
\end{corollary}

\subsection{Foliated affine structures}  A foliated  affine structure is equivalent to the data of a foliated connection \(\nabla\) which to each vector field  \(Z\) tangent to the foliation assigns its Christoffel symbol \(\Gamma(Z)\), defined by \(\nabla_Z  Z=\Gamma(Z) Z\). It satisfies the relation \(\Gamma(fZ)=f\Gamma(Z)+Zf\) \cite[Section~2.1.2]{DG}.  A singularity \(p\) of \(\mathcal{F}\) is said to be \emph{parabolic} if for every (or for one)   vector field \(Z\) with singular set of codimension at least two defining \(\mathcal{F}\) in a neighborhood of \(p\),  \(\Gamma(Z)|_p=0\).  

We have the following analogue of Theorem~\ref{thm:proj_lin} for foliated affine structures.

\begin{thm}\label{thm:aff_lin}  Let \(\mathcal{F}\) be a foliation on a manifold endowed with a uniformizable foliated affine structure, \(p\) be a non-parabolic singularity of \(\mathcal{F}\). There exist coordinates around \(p\) in which \(\mathcal{F}\) is generated by a linear diagonal vector field  with integral eigenvalues, whose affine Christoffel symbol is constant and equal to~\(1\).
\end{thm}

Let us indicate the lines of a first proof of this result. Let \(U\) be a neighborhood of~\(0\) in \(\CC^n\), \(Z\) a vector field on \(U\) having a singularity at \(0\), and \(\mathcal{F}\)  the foliation defined by \(Z\). Consider a uniformizable foliated affine structure on \(\mathcal{F}\). On \(U\times\CC\), consider the \emph{geodesic vector field} of the foliated affine structure,
\(X=\zeta Z-\Gamma(Z) \zeta^2\,\indel{\zeta}\), together with its  companion vector field \(H=\zeta\,\indel{\zeta}\), which is complete and \(2\ii\pi\)-periodic, and which satisfies, with \(X\), the Lie bracket relation \([H,X]=X\) (see \cite[Section~4.1]{DG} for details). After establishing that the uniformizability of the foliated affine structure implies that the geodesic vector field is semicomplete, the proof follows along the lines  of that of Theorem~\ref{thm:proj_lin}. 

Theorem~\ref{thm:aff_lin} can also be established as an application of Theorem~\ref{thm:proj_lin}. A foliated affine structure on \(\mathcal{F}\) may be seen as a foliated projective structure: if \(\Gamma\) is its affine Christoffel symbol, the projective Christoffel symbol \(\Xi\) for the induced foliated projective structure is given by \(\Xi(Z)=-\frac{1}{2}\Gamma(Z)^2+Z\Gamma\) \cite[Section~2.2]{DG}. In the opposite direction, for a foliated projective structure with Christoffel symbol \(\Xi\), if, for a vector field \(Z\), tangent to \(\mathcal{F}\), there exists a function \(\gamma:U\to\CC\) such that \(Z\gamma= \frac{1}{2}\gamma^2+\Xi(Z)\), the projective structure comes from an affine one: the foliated affine structure on \(U\)  given by the affine Christoffel symbol \(\Gamma(Z)=\gamma\) is in the projective class of \(\Xi\). A foliated affine structure is uniformizable if and only if it is uniformizable as a projective one, and a  singular point will be parabolic for a foliated affine structure if and only if it is parabolic when the structure is considered as a projective one. In order to prove Theorem~\ref{thm:aff_lin}, we may consider the foliated affine structure as a projective one, apply Theorem~\ref{thm:proj_lin}, and then come back to the foliated affine structure.

As an application of Theorem~\ref{thm:aff_lin}, we have the following result. It concerns semicomplete meromorphic vector fields whose divisors of zeros and poles have normal crossings and are tangent to the induced foliation, a central situation in the birational theory of semicomplete vector fields (compare with \cite[Prop.~5.4]{rebelo-nonisolees} and \cite[Props.~17 and~18]{guillot-rebelo}; see also \cite[Lemma~2]{delarosa-parameter}). 
\begin{prop}
Let  
\[X=\sum_{i=1}^k z_i h_i\del{z_i}+\sum_{i>k} g_i\del{z_i}\]
be a vector field on a neighborhood of \(0\) in \(\CC^n\), with \(h_i\) and \(g_i\) holomorphic for all \(i\), vanishing at \(0\), with singular set of codimension at least two.  If, for the integers \(p_1, \ldots, p_k\), the germ of \(Z=z_1^{p_1}\cdots z_k^{p_k}X\) at \(0\) is semicomplete (if there is a neighborhood of \(0\) in which, away from its divisor of poles, all of its solutions are univalent), and \(\sum_{i=1}^k p_ih_i(0)\neq 0\), there exist coordinates centered at \(0\) in which \[Z=z_1^{p_1}\cdots z_k^{p_k} \sum_{i=1}^n \lambda_iz_i \del{z_i},\]
for integers \(\lambda_1, \ldots, \lambda_n\) such that \(\sum_{i=1}^k \lambda_i p_i= -1\). 
\end{prop}
\begin{proof} The semicompleteness of \(Z\) is equivalent to the uniformizability of the  foliated affine structure that it induces \cite[Prop.~9]{guillot-rebelo} (we refer the reader to \cite{guillot-rebelo} for a general discussion about semicomplete meromorphic vector fields and their relation to foliated affine structures).  For the foliated affine structure induced by \(Z\), from \(\Gamma(Z)\equiv 0\)  and (\ref{eq:leibniz}), for \(f=z_1^{p_1}\cdots z_k^{p_k}\),
\[\Gamma(X)=-\frac{1}{f}Xf=-\sum_{j=1}^k p_ih_i.\]
The standing hypothesis \(\sum_{i=1}^k p_ih_i(0)\neq 0\) is thus equivalent to the fact that the foliated affine structure is not parabolic at the singular point.
By Theorem~\ref{thm:aff_lin} (and its proof via Theorem~\ref{thm:proj_lin}), there exist coordinates where 
\[Z=u(z_1,\ldots,z_n)z_1^{p_1}\cdots z_k^{p_k}Y,\]
for  \(u\) a non-vanishing holomorphic function, \(Y=\sum_{i=1}^n \lambda_iz_i\, \indel{z_i}\), with \(\lambda_i\in \ZZ\), and \(\Gamma(Y)\equiv 1\) (one needs to observe that, in this case, the change of coordinates~(\ref{bochner-cartan})  in the proof of Theorem~\ref{thm:proj_lin}  preserves the invariant hyperplanes coming from \(\{z_i=0\}\) for \(i=1, \ldots, k\)). From \(\Gamma(Z)\equiv 0\), we have that
\(u^{-1}Yu+\sum_{i=1}^k \lambda_i p_i+1=0\), and, since \((Yu)(0)=0\), it follows that  \(\sum_{i=1}^k \lambda_i p_i=-1\), and that, in consequence, \(Yu\equiv 0\). Let \(w_i=u^{-\lambda_i}z_i\) for \(i\leq k\), \(w_i=z_i\) for \(i>k\). Since
\begin{multline*}\frac{1}{w_i}Z w_i=\lambda_i u \left(\prod_{j=1}^k z_j^{p_j}\right)=\lambda_iu\left(\prod_{j=1}^k w_j^{p_j}u^{\lambda_j p_j}\right) = \\ = \lambda_i u^{1+\lambda_1 p_1+\cdots + \lambda_k p_k}\left(\prod_{j=1}^k w_j^{p_j}\right) 
=\lambda_i\left(\prod_{j=1}^k w_j^{p_j}\right),\end{multline*}
the vector field has, in the coordinates \((w_1,\ldots,w_n)\), the desired form.
\end{proof}

\section{Global consequences} 
Theorem~\ref{thm:rat1stint} will be proved in this section. We start with a semilocal version of it, which promotes Corollary~\ref{coro:norforstrong} to a description of  a saturated neighborhood of a non-degenerate, non-parabolic singular point of a strongly uniformizable foliated projective structure.

\begin{prop}\label{prop:local_strong} Let \(M\) be a manifold of dimension \(k+1\), \(\mathcal{F}\) a holomorphic foliation by curves on \(M\) endowed with a strongly uniformizable foliated projective structure, \(p\in M\) a non-degenerate singularity of \(\mathcal{F}\) which is non-parabolic for the foliated projective structure. Then, either
\begin{itemize}
\item \(M\) is biholomorphic to \(\PP^{k+1}\), with \(\mathcal{F}\) corresponding to the foliation by lines through \(p\); or
\item there exists an open saturated subset \(\Omega \subset M\) containing \(p\), a pointed analytic space \((Y,o)\) of dimension \(k\), and a holomorphic map \(f:(\Omega,p)\to(Y,o) \) such that \(f|_{\Omega\setminus f^{-1}(o)}:\Omega\setminus f^{-1}(o)\to  Y\setminus\{o\}\) is a proper onto map  whose fibers are rational curves invariant by \(\mathcal{F}\).
\end{itemize}
\end{prop}	 
\begin{proof} Through Corollary~\ref{coro:norforstrong}, suppose that, in suitable coordinates centered at \(p\), the foliation is generated by the vector field \(Z=\sum_{i=1}^{m} z_i\,\indel{z_i}-\sum_{i=m+1}^{m+n} z_i\,\indel{z_i}\) (\(m+n=k+1\)). Suppose, without loss of generality,  that \(m>0\). Restrict to
\[B=\{z \mid |z_1|^2+\cdots+ |z_m|^2<\epsilon^2,\;  |z_{m+1}|^2+\cdots+ |z_{m+n}|^2 <\epsilon^2\}.\]
Let \(W^-_\mathrm{loc}=\cap_{i\leq m}\{z_i=0\}\), \(W^+_\mathrm{loc}=\cap_{i> m}\{z_i=0\}\). The restriction of \(Z\) to \(W^+_\mathrm{loc}\) is \(\sum_{i=1}^{m}  z_i\,\indel{z_i}\), and there is thus a local invariant curve for \(\mathcal{F}\) by \(p\) through each direction tangent to \(W^+_\mathrm{loc}\). By Lemma~\ref{lemma:completecurve}, each one of these curves extends to a rational curve in \(M\) meeting no other singular point of \(\mathcal{F}\). The saturated of \(W^+_\mathrm{loc}\) is thus a compact manifold \(W^+\subset M\), biholomorphic to \(\PP^m\), on which the restriction of \(\mathcal{F}\) is equivalent to the foliation by lines through a point. This proves our claim when \(n=0\). 

Let us henceforth suppose that \(n> 0\). There is a compact invariant manifold \(W^-\), biholomorphic to \(\PP^{n}\), extending \(W^{-}_\mathrm{loc}\). For \(z\in B\setminus (W^+_\mathrm{loc}\cup W^-_\mathrm{loc})\), consider the annulus
\[ A_z=\left\{t\in\CC\;\middle|\;  \frac{|z_{m+1}|^2+\cdots+ |z_{m+n}|^2}{\epsilon^2} <|t|^2<  \frac{\epsilon^2}{ |z_1|^2+\cdots+ |z_m|^2}  \right\},\]
which contains \(\{|t|=1\}\). Let \(L_z\) be the leaf of \(\mathcal{F}\)  through \(z\). The map \(f_z:A_z\to L_z\),
\[t\mapsto (z_1t,\ldots,z_mt,z_{m+1}t^{-1},\ldots,z_{m+n}t^{-1}),\] 
gives a one-to-one parametrization of \(L_z\cap B\), which is thus an annulus as well.  We have that \(\mathrm{D}f_z(t\,\indel{t})=Z|_{L_z}\), and, from \(\Xi(t\,\indel{t}) \equiv -1/2\) and formula (\ref{eq:leibniz}), \(\Xi(\indel{t}) \equiv 0\): the projective structure on the image of \(A_z\) is the one it has as a subset of \(\CC\).

The boundary  of \(B\) has two components: \(\partial^+B=\{\sum_{i\leq m} |z_i|^2=\epsilon^2\} \) and \(\partial^- B= \{\sum_{i>m} |z_i|^2=\epsilon^2\}\).
One of the  boundary components of \(L_z\cap B\) is contained in \(\partial^+B\), the other in \(\partial^-B\). Let \(L\subset W^+\) be a leaf of \(\mathcal{F}\). It intersects \(\partial^-B\) along a circle, and  has trivial monodromy. If \(L_z\) is a sufficiently close leaf that is not contained in \(W^+\), it intersects \(\partial^-B\) along a neighboring circle. Outside \(B\), \(L_z\) follows \(L\), and consists of gluing a disk to  the boundary component \(L_z\cap \partial^-B\) of the annulus \(L_z\cap \overline{B}\). The same happens close to \(W^-\). The leaves of \(F\) close to \(W^+\cup W^-\) are thus compact rational curves which do not   intersect the singular set of the foliation. Each one of them consists of an annulus (its trace in \(B\)) plus a disk glued to each one of its boundary components.

For \(i\leq m\) and \(j>m\), consider the function \(f_{ij}:B\to \CC\),  \(f_{ij}=z_iz_j\). Let \(f:B\to \CC^{nm}\) be the map defined by the \(f_{ij}\), \(Y\subset \CC^{nm}\) its image. Let us show that this first integral \(f:B\to Y\) of \(\mathcal{F}\) separates orbits that are not in \(W^+\cup W^-\). Let \(p\) and \(q\) be points in the complement of  \(W^+\cup W^-\) such that \(f(p)=f(q)\). If \(z_1(p)\) and \(z_1(q)\) are both nonzero, we may suppose, up to the action of \(Z\) on \(p\), that they are equal. Since \(f_{1i}(p)=f_{1i}(q)\) for every \(i>m\),  for every \(i> m\) we have that \(z_i(p)=z_i(p)\). Since at least one of these is not zero, the same argument shows that, for every \(j\leq m\), \(z_j(p)=z_j(q)\). By extending \(f\) to the saturated of \(B\), we obtain the result. \end{proof}

\begin{proof}[Proof of Theorem~\ref{thm:rat1stint}] Through Theorem~\ref{thm:inv_leavess} (Appendix~\ref{app:cp}),   Proposition~\ref{prop:local_strong} implies that for every regular point \(p\) of  \(\mathcal{F}\) there is a closed curve containing the leaf of  \(\mathcal{F}\) through \(p\); the result then follows directly from Theorem~\ref{thm:gm_kahler}.  \end{proof}

\begin{rmk}\label{rmk:fi_surfaces} In Theorem~\ref{thm:rat1stint}, the  K\"ahlerness of the ambient manifold is not necessary in dimension two: from Proposition~\ref{prop:local_strong}, \(\mathcal{F}\) has infinitely many compact leaves, and in this setting, since the foliation is of codimension one, a theorem of Jouanolou and Ghys \cite{ghys-jouanolou} guarantees the existence of a meromorphic first integral.\end{rmk}

A weak form of Theorem~\ref{thm:rat1stint}, not requiring the K\"ahler hypothesis, follows from a result in Barlet's preprint \cite{barlet:hal-01175721} (Proposition~\ref{prop:local_strong} gives the setting of Barlet's Corollary~1.0.2). It follows from it that, if we are not in the first possibility of Proposition~\ref{prop:local_strong}, the ``stability region,'' the open subset formed by the smooth compact rational leaves that do not intersect the singular locus of the foliation, is dense.

\section{The index theorem} 

Theorem~\ref{thm:index_parab} is an extension of Theorem~5.1 in \cite{DG}. From the foliated projective structure, we will consider a foliation by curves on a compact \((n+1)\)-dimensional manifold,  and will apply the Baum-Bott index theorem to it.  

\begin{proof}[Proof of Theorem~\ref{thm:index_parab}]  We begin by describing a \(\PP^1\)-bundle \(N\) over \(M\) through some explicit formulas. Let \(\{U_i\}\) be an open cover of \(M\) and \(Z_i\) a vector field on \(U_i\) tangent to \(\mathcal{F}\). Let \(\{g_{ij}\}\) be the associated cocycle, so that \(Z_i = g_{ij}Z_j\) in \(U_i\cap U_j\). In \(\{U_i\times \PP^1\}\), identify \(U_i\times \PP^1\) with \(U_j\times \PP^1\) via the natural identification of \(U_i\) and \(U_j\) along their intersection and, in the second factor, through \(u_i=g_{ij}u_j-Z_jg_{ij}\), where \(u_i\) and \(u_j\) are affine coordinates in the corresponding projective lines. The resulting compact manifold \(N\) is a \(\PP^1\)-bundle \(\pi:N\to M\).
	
We will now use the foliated projective structure to define a foliation by curves \(\mathcal{G}\) on \(N\), the projectivization of the projective geodesic vector field that already appeared in Section~\ref{sec:normal_forms}. Let  \(\rho_i:U_i\to\CC\) be the projective Christoffel symbol of \(Z_i\). Consider the foliation \(\mathcal{G}_i\) on \(U_i\times \PP^1\) given by the vector field
\[W_i=Z_i-\left({\textstyle\frac{1}{2}}u_i^2+\rho_i\right)\del{u_i}.\] These foliations glue together and define a global foliation \(\mathcal{G}\) on \(N\), which  projects to \(\mathcal{F}\) via~\(\pi\).

We will apply the Baum-Bott index theorem to \(\mathcal{G}\). It  establishes that the number resulting from the evaluation of \(\varphi\) on the Chern classes of the (virtual) normal bundle \(TN-T_\mathcal{G}\) of \(\mathcal{G}\) can be localized at the singular set of \(\mathcal{G}\), and computed through some residues \cite{baum-bott}.  From the proof of  Theorem~5.1 in \cite{DG}, \(\varphi(c(TN-T_\mathcal{G}))\) equals twice the expression appearing in the left hand side of (\ref{eq:index_qf}) plus
\begin{equation}\label{eq:index_lhs}2\widehat\varphi_n(c(TM-T_\mathcal{F})),\end{equation}
the term  that would correspond to \(k=0\) in~(\ref{eq:index_qf}).
	
Let us  calculate the local contributions of the singularities. Above each singular point of \(\mathcal{F}\), there is a unique singular point of \(\mathcal{G}\), which is degenerate. For \(\epsilon\) in a neighborhood of \(0\) in \(\CC\), consider the foliation on \(U_i\times \PP^1\) given  by \(W_i+\epsilon\, \indel{u_i}\), which gives the foliation \(\mathcal{G}_i\) when \(\epsilon=0\). The singularity of \(\mathcal{G}_i\) above \((p,0)\) splits, when \(\epsilon\neq 0\), into two non-degenerate ones. One of them has the eigenvalues \(\lambda_1,\ldots,\lambda_n\) of \(Z_i\) plus the eigenvalue \(\mu(\epsilon)=\sqrt{2(\epsilon-\rho_i)}\), vanishing with \(\epsilon\); the other, the eigenvalues \(\lambda_1,\ldots,\lambda_n\) and \(-\mu(\epsilon)\). The sum of the Baum-Bott residues with respect to \(\varphi\) of these two is
\[\frac{\varphi(\lambda_1,\ldots,\lambda_n,\mu)}{\lambda_1\cdots\lambda_n\mu}-\frac{\varphi(\lambda_1,\ldots,\lambda_n,-\mu)}{\lambda_1\cdots\lambda_n\mu} =\frac{2}{\lambda_1\cdots\lambda_n}\sum_{j \text{ odd}}\mu^{j-1}\widehat{\varphi}_{n+1-j}(\lambda_1,\ldots,\lambda_n).\]
According to \cite[Thm.~5.4]{bracci-suwa}, the residue of the singularity of \(\mathcal{G}_i\) above \(p\) is the limit of this expression as \(\epsilon\) goes to \(0\), which equals
\begin{equation}\label{eq:index_rhs} 2\frac{\widehat{\varphi}_{n}(\lambda_1,\ldots,\lambda_n)}{\lambda_1 \cdots \lambda_n}.\end{equation}
By the Baum-Bott index theorem applied to \(\mathcal{F}\) via the symmetric polynomial \(\widehat{\varphi}_n\), the expressions (\ref{eq:index_lhs})  and (\ref{eq:index_rhs}) are equal, and cancel each other. This establishes the result. \end{proof}

\begin{rmk}\label{stat:simple} For  even \(n\), for \(\varphi=\sum_{i=1}^{n+1}x_i^{n+1}\), formula (\ref{eq:index_qf}) reduces to \(c_1^n(T_\mathcal{F})=0\); For odd \(n\), for \(\varphi=\sum_{i=1}^{n+1}x_i^{n}\sum_{j\neq i} x_j\), it reads \(c_1^{n-1}(T_\mathcal{F})c_1(TM-T_\mathcal{F})=0\).
\end{rmk}

Theorem~\ref{thm:index_parab} imposes severe restrictions on foliations on manifolds of even dimensions.  We have the following extension of Proposition~6.3 in \cite{DG}:

\begin{corollary} Let \(M\) be a compact complex algebraic manifold of even dimension, \(\mathcal{F}\) a holomorphic foliation on \(M\) endowed with a  foliated projective structure. Suppose that the singularities of \(\mathcal{F}\) are non-degenerate  and non-dicritical, and that the foliated projective structure makes them all parabolic.  Then, \(\mathcal{F}\) is not of general type.
\end{corollary}
(A non-degenerate singularity of a foliation is \emph{dicritical} if the ratios of the eigenvalues of the linear part of a vector field tangent to it are all positive rationals; a foliation is said to be \emph{of general type} if the Iitaka dimension of its canonical bundle  is the dimension of the ambient manifold.)

\begin{proof}
If \(K_\mathcal{F}\) is not nef then, by a result by Bogomolov and McQuillan \cite[Cor.~4.2.1]{bogomolov-mcquillan}, either the foliation is a ruling by rational curves (in which case it is not of general type, for no positive power of its canonical bundle has non-zero sections), or there exists an invariant rational curve \(C\) in \(M\) such that \(K_\mathcal{F}\cdot C=-1\). In this case, (\ref{eq:ph_inv}) reads \(\mu(\mathcal{F}, C )=1\). This implies that \(C\) intersects the singular set of \(\mathcal{F}\) along a single branch passing through a singular point \(p\) of \(\mathcal{F}\), and that no other singular point of \(\mathcal{F}\) lies on \(C\). Since \(p\) is, by hypothesis, a parabolic singular point, the ramification index at \(p\) of the projective structure induced on \(C\) is \(\infty\), and the local monodromy is parabolic.  But this is impossible, for  \(C\setminus \{p\}\), along which the projective structure is regular, is simply connected. This  discards the case in which \(K_\mathcal{F}\) is not nef. If \(K_\mathcal{F}\) is nef then, as in the proof of Proposition~6.3 in \cite{DG}, from the asymptotic Riemann-Roch formula, the condition on \(c_1(K_\mathcal{F})\) arising from Remark~\ref{stat:simple} is incompatible with   \(K_\mathcal{F}\) being of general type. \end{proof}

\begin{proof} [Proof of Corollary~\ref{coro:unif-typegeneral}] If all the singular points of the foliation are parabolic, the result follows from the previous corollary. Otherwise,  there is a non-parabolic singular point, and, by Theorem~\ref{thm:rat1stint} and Proposition~\ref{prop:local_strong}, the foliation is a ruling by rational curves, with an open subset foliated by compact curves that do not intersect the singular locus of \(\mathcal{F}\), and \(\mathcal{F}\) is not of general type.
\end{proof}

\begin{rmk} In the proof of Proposition~6.3 in \cite{DG}, it should have been mentioned that, for a regular holomorphic foliation by curves on a compact manifold, either its canonical bundle is nef, or the foliation is a fibration by rational curves \cite[Cor.~4.2.2]{bogomolov-mcquillan}, and that in the second case, the foliation is not of general type. 
\end{rmk}

\section{Foliations on projective spaces} \label{sec:proj_space}

In order to prove  Theorem~\ref{thm:pn_nounif}, we begin by stating the following result,  which follows straightforwardly from Theorem~\ref{thm:proj_lin} and Lemma~\ref{lemma:completecurve}:
\begin{lemma}\label{lemma:ratcurve} Let \(M\) be a manifold, \(\mathcal{F}\) a holomorphic foliation by curves on \(M\) endowed with a strongly uniformizable projective structure, and \(p\in M\) a non-degenerate, non-parabolic singularity of \(\mathcal{F}\). Then, there exists a non-singular rational curve \(C\subset M\), invariant by \(\mathcal{F}\), with \(\mathrm{sing}(\mathcal{F})\cap C=\{p\}\).
\end{lemma}

\begin{proof}[Proof of Theorem~\ref{thm:pn_nounif}] Let \(\mathcal{F}\) be a foliation of degree \(d\) on \(\PP^n\), so that \(K_\mathcal{F}=(d-1)h\), for \(h\) the hyperplane class in \(H^2(\PP^n,\ZZ)\). If \(p\) is  a non-degenerate, non-parabolic singularity of \(\mathcal{F}\) then, by Lemma~\ref{lemma:ratcurve},  there exists a smooth rational curve \(C\), invariant by \(\mathcal{F}\), intersecting   \(\mathrm{sing}(\mathcal{F})\) exclusively at \(p\). If \(\delta\) denotes its degree, formula (\ref{eq:ph_inv}) gives \((d-1)\delta=-1\), which implies that \(d=0\).  Thus, if \(d>0\), all the singularities of \(\mathcal{F}\) are   parabolic, and \(\mathcal{F}\) is subject to the conditions imposed by Theorem~\ref{thm:index_parab}. If \(n\) is even, by Remark~\ref{stat:simple},  \((1-d)^n=0\), and \(d=1\). If \(n\) is odd, from
\[c(T\PP^n-T_\mathcal{F})=\frac{c(T\PP^n)}{c(T_\mathcal{F})}=\frac{(1+h)^{n+1}}{1-(d-1)h}=1+(n+d)h+\cdots,\]
\(c_1(T\PP^n-T_\mathcal{F})=(n+d)h\), and Remark~\ref{stat:simple} gives \((1-d)^{n-1}(n+d)=0\); we also have \(d=1\) in this case.
\end{proof}

In Theorem~\ref{thm:pn_nounif}, the hypothesis on the degree is a necessary one. A foliation of vanishing degree on \(\PP^n\) is given by the lines through a point. Its leaves are biholomorphic to \(\CC\), and the canonical projective structure on each one of them gives  a strongly uniformizable foliated projective structure. Foliations of degree one on \(\PP^n\) are induced by globally defined holomorphic vector fields, which are complete, and which induce strongly uniformizable foliated translation (thus projective) structures. We do have foliations on \(\PP^2\) of degree two with non-degenerate singularities supporting uniformizable foliated affine (hence projective)  structures (which are not, of course, strongly uniformizable); see~\cite{guillot-sigma}.

The results in Theorem~\ref{thm:pn_nounif} coexist with an abundance of foliated projective structures along the involved foliations. For a foliation by curves \(\mathcal{\mathcal{F}}\), the space of foliated projective structures subordinate to it, if not empty, is an affine space directed by the vector space of quadratic differentials along the leaves of \(\mathcal{F}\) \cite[Rmk.~2.15]{DG}. For a foliation by curves \(\mathcal{F}\) of degree \(d\) on \(\PP^n\), this space is not empty~\cite[Ex.~2.4]{DG}, and, if \(d\geq 1\), \(K^2_{\mathcal{F}}=(T^2_{\mathcal{F}})^*=O(2(d-1))\). This space may be identified	with the space of homogeneous polynomials of degree \(2(d-1)\) in \(n+1\) variables, and thus \[h^0(\PP^n,K^2_{\mathcal{F}})=\binom{n+2d-2}{n}.\]
In particular, for \(n=2\), \(h^0(\PP^2,K^2_{\mathcal{F}})=d(2d-1)\), which, if \(d>2\), is strictly greater than  (and roughly twice as big as) \(d^2+d+1\), the number  of singularities of a generic foliation of degree \(d\) (at a non-degenerate singular point, parabolicity is, locally, a codimension-one condition, and one may be na\"ively tempted to think that, for a given foliation with non-degenerate singularities, and sufficiently many foliated projective structures, some of the latter make all  the singular points parabolic).

\appendix

\section{On the analytic subspaces invariant by  foliations on K\"ahler manifolds}\label{app:cp}

A holomorphic foliation of codimension one on a compact complex manifold with  infinitely many compact invariant hypersurfaces has a meromorphic first integral, and the union of the invariant hypersurfaces is  the whole manifold (this result is due Ghys \cite{ghys-jouanolou}; it generalizes previous work by Jouanolou). This is no longer true for foliations of higher codimension, as observed  in the closing remarks to \cite{ghys-jouanolou}.  In its place, we have, at least in the K\"ahler setting, the following result:

\begin{thm}[Pereira]\label{thm:inv_leavess} Let \(M\) be a compact K\"ahler manifold, \(\mathcal{F}\) a (not necessarily regular) holomorphic foliation on \(M\) of dimension \(d\). The union of the closed analytic subspaces of \(M\) of dimension \(d\) that are invariant by \(\mathcal{F}\) is the union of countably many closed analytic subspaces.\end{thm}

This result has a natural counterpart in the algebraic setting, which is implicit in the proof of Theorem~3 in \cite{gomezmont-integrals}, and of which Proposition~2.1 in \cite{coutinho-pereira} is a relative version. The proof we present is due to Jorge Pereira, and some of its arguments are not far from those in \cite{gomezmont-integrals}. The ideas behind it will allow us to establish a K\"ahlerian version of a result by G\'omez-Mont \cite[Thm.~3]{gomezmont-integrals}, the final remark in \cite{gomezmont-integrals} notwithstanding:
\begin{thm}[Pereira]\label{thm:gm_kahler} Let \(\mathcal{F}\) be a (not necessarily regular) holomorphic foliation of codimension \(q\) in the compact complex K\"ahler manifold \(M\), and assume that through every regular point of \(\mathcal{F}\) there is an invariant compact analytic space of codimension \(q\) invariant by \(\mathcal{F}\). Then, there exists an analytic space \(V\), of dimension \(q\), belonging to Fujiki's class \(\mathcal{C}\), and a dominant meromorphic map \(f : M \dashrightarrow V\) such  that the leaves of \(\mathcal{F}\) are contained in the fibers of~\(f\).  
\end{thm}

Here, a holomorphic foliation \(\mathcal{F}\) of dimension \(d\) on a complex manifold \(M\) is a coherent analytic subsheaf \(\mathcal{F}\subseteq TM\) of rank \(d\) satisfying the Frobenius integrability condition \([\mathcal{F},\mathcal{F}]\subseteq \mathcal{F}\) (see \cite{bb72}). An analytic subspace \(S\subset M\) defined by the sheaf of ideals \(\mathcal{J}_S\subset \mathcal{O}(M)\) is invariant by \(\mathcal{F}\) if and only if \(\mathcal{F}\cdot\mathcal{J}_S\subset \mathcal{J}_S\), this is, if for every vector field \(Z\) in \(\mathcal{F}\) and every \(f\in \mathcal{J}_S\), \(Zf\in \mathcal{J}_S\). An invariant variety whose dimension is strictly smaller than that of \(\mathcal{F}\) is necessarily  contained in \(\mathrm{sing}(\mathcal{F})\).

The central ingredient in the proofs of these results is the \emph{Douady space} of the manifold \(M\), a complex space which parametrizes flat families of compact analytic subspaces of \(M\). We refer the reader to \cite{campana-peternell} for a discussion around this space as well as of some of its properties, which will be used without further comments.
	
We begin by establishing a preliminary result (compare with \cite[Lemma~2]{gomezmont-integrals}).

\begin{lemma} \label{lemma:app}	
Let \(M\) be a compact complex manifold, \(\mathcal{F}\) a foliation of dimension \(d\)  on \(M\). Let \(T\) be a compact complex analytic space, and \(S\subset M\times T\) an analytic subspace such that the induced projection \(\pi:S\to T\) is equidimensional with fibers of dimension \(d\). Then,  
\begin{enumerate} 
\item \label{item1:lemma:app} the subset of \(T\) corresponding to fibers of \(\pi\) having an irreducible component invariant by \(\mathcal{F}\), and 
\item \label{item2:lemma:app} the subset of \(M\) formed by points that belong to an irreducible component of a fiber of \(\pi\) invariant by \(\mathcal{F}\)
\end{enumerate}
are closed analytic subsets.
\end{lemma}

\begin{proof} Let \(\widehat{\mathcal{F}}\) be the foliation of dimension \(d\) on \(M\times T\) tangent to the fibers of \(\pi\),  and coinciding with \(\mathcal{F}\) along them. Let \(\mathcal{J}_S\subset \mathcal{O}(M\times T)\) be the sheaf of ideals defining~\(S\). Define recursively the ascending chain of sheaves of ideals \(\mathcal{J}^0_S\subseteq \mathcal{J}^1_S\subseteq \mathcal{J}^2_S\subseteq \cdots \) by setting \(\mathcal{J}^0_S=\mathcal{J}_S\), and defining \(\mathcal{J}^{i+1}_S\) as the sheaf generated by \(\mathcal{J}^i_S\) and the derivatives of the elements of \(\mathcal{J}^i_S\) along the vector fields tangent to \(\widehat{\mathcal{F}}\). By local Noetherianity and the compactness of \(M\times T\), the chain of ideals   stabilizes: there exists \(n\in\NN\) such that \(\mathcal{J}^{n+1}_S=\mathcal{J}^n_S\). Thus, for every \(f\in \mathcal{J}^n_S\) and every vector field \(Z\) tangent to \(\widehat{\mathcal{F}}\), \(Zf\in \mathcal{J}^n_S\), and \(\mathcal{J}^n_S\) defines an  analytic subspace \(S^{ \mathcal{F}}\subset M\times T\), invariant by \(\widehat{\mathcal{F}}\) and contained in \(S\).  
	
If \(X\subset M\) is an  analytic subspace invariant by \(\mathcal{F}\), and \(X\times \{t_0\}\subset S\), then if  \(X\times \{t_0\}\) is defined by the sheaf of ideals \(\mathcal{J}_X\subset \mathcal{O}(M\times T)\), since \(\mathcal{J}_S\subset \mathcal{J}_X\), and \(\widehat{\mathcal{F}}\cdot \mathcal{J}_X\subset \mathcal{J}_X\), we have that \(\mathcal{J}^i_S\subset \mathcal{J}_X\) for all \(i\), and \(X\times \{t_0\}\) is thus contained in \(S^{\mathcal{F}}\). Reciprocally, if \(X\subset M\) is an  analytic subspace, and \(X\times\{t_0\}\subset S'\), \(X\) is invariant by \(\mathcal{F}\). We also have that \(S\cap\mathrm{sing}(\widehat{\mathcal{F}})=S\cap(\mathrm{sing}(\mathcal{F})\times T)\) is contained in~\(S^{\mathcal{F}}\).

Let \(S'\) be the subset of \(S^\mathcal{F}\) formed by the \(x\) such that the dimension at \(x\) of \(\pi^{-1}\circ \pi(x)\cap S^\mathcal{F}\) is \(d\). By the semicontinuity of the fiber dimension \cite[Ch.~3, \S3.4]{fischer}, \(S'\) is a closed analytic subset of \(S^\mathcal{F}\) (and of~\(S\)). 

Let \(T'=\pi(S')\). It is the subset of \(T\) corresponding to the fibers of \(\pi\) having an irreducible component invariant by \(\mathcal{F}\). By the Proper Mapping Theorem \cite[Ch.~10, \S6]{grauert-remmert}, it is a closed analytic subspace of \(T\). This establishes item~(\ref{item1:lemma:app}).

The image of \(S'\) in \(M\) under the projection \(M\times T\to M\) is the set of item~(\ref{item2:lemma:app}). Again, by the Proper Mapping Theorem, it is a closed analytic subset of \(M\). This establishes item~(\ref{item2:lemma:app}).\end{proof}

\begin{proof}[Proof of Theorem~\ref{thm:inv_leavess}] Consider the Douady space \(\mathcal{D}(M)\) of \(M\). Since the ambient manifold is K\"ahler, each one of its  irreducible components is compact, and belongs to Fujiki's class \(\mathcal{C}\).  Let \(T\) be an irreducible component of \(\mathcal{D}(M)\) parametrizing analytic subspaces of dimension \(d\), and let \(S\subset M\times T\) be the restriction of the universal family to \(T\). Let \(T'\subset T\) and \(S'\subset M\times T'\) be  as in Lemma~\ref{lemma:app} (\(T'\) is also in Fujiki's class \(\mathcal{C}\)). The  image of \(S'\) under the projection onto \(M\) is, by  item (\ref{item2:lemma:app}) in Lemma~\ref{lemma:app}, a closed analytic subspace of \(M\). As \(T\) varies within the countably many irreducible components of \(\mathcal{D}(M)\), we obtain the countable union of analytic subspaces of \(M\) in the statement of the theorem. \end{proof}

\begin{proof}[Proof of Theorem~\ref{thm:gm_kahler}] We follow the closing argument in G\'omez-Mont's proof of Theorem~3 in \cite{gomezmont-integrals}. Under the standing hypothesis, by Theorem~\ref{thm:inv_leavess}, there is a component \(T\) of the  Douady space of \(M\) such that, with the previous notations, the restriction to \(S'\)  of the projection \(M\times T'\to M\) is a holomorphic onto map. This map is injective in the complement of a closed proper analytic subset (for there is only one leaf passing through a regular point of \(\mathcal{F}\)), and is thus a bimeromorphic map, having an inverse \(\phi:M\dashrightarrow S'\). The meromorphic map \(\pi'|_{S'}\circ\phi:M\dashrightarrow   T'\) is the sought one. \end{proof}

\begin{rmk} Together, Theorems~\ref{thm:inv_leavess} and~\ref{thm:gm_kahler} extend the main result in \cite{pereira-global} to encompass  singular foliations. This is essentially what allowed us, in the proof of Theorem~\ref{thm:rat1stint}, to go  from Proposition~\ref{prop:local_strong} to a setting where Theorem~\ref{thm:gm_kahler} could be used.\end{rmk}

We end by exhibiting some regular holomorphic foliations  on algebraic compact  manifolds with infinitely many compact leaves but without first integrals; not many examples of these seem to be known. In them, the union of the closed leaves is not contained in the union of finitely many proper closed analytic subspaces.

\begin{example}[Complex geodesic foliations of complex hyperbolic surfaces with infinitely many closed complex geodesics] Let \(q\) be a pseudo-Hermitian form on \(\CC^3\) of signature \((1,2)\), and let \(\Gamma\subset \mathrm{PU}(q)\) be a cocompact lattice. As  discussed in the beginning of Section~\ref{sec:geod-fol}, on the projectivized tangent bundle of the complex hyperbolic surface \(\Gamma\backslash \BB^2\), we have the complex geodesic foliation (a non singular holomorphic foliation by curves); its  compact leaves correspond to the closed complex geodesics in~\(\Gamma\backslash \BB^2\).
	
There exist plenty of forms \(q\) as above for which \(\mathrm{PU}(q)\) admits cocompact lattices, and for which the associated compact complex hyperbolic manifold  has infinitely many closed complex geodesics; see \cite[Section~1]{moller-toledo}. The simplest of these are the pseudo-Hermitian forms   defined over \(\QQ\), of signature \((1,2)\) that do not represent~\(0\). For these, \(\mathrm{P}(\mathrm{U}(q)\cap\mathrm{SL}(3,\ZZ)) \subset \mathrm{PU}(q)\) is a cocompact lattice, and the countably many rational planes of \(\CC^3\) in restriction to which \(q\) has signature \((1,1)\) give closed complex geodesics in~\(\Gamma\backslash \BB^2\).\end{example}

\subsection*{Acknowledgments} We thank Xavier G\'omez-Mont and Jorge Pereira for  enlightening conversations around the results discussed in Appendix~\ref{app:cp}. We are also grateful to Jorge Pereira for having generously allowed us to include Theorems~\ref{thm:inv_leavess} and~\ref{thm:gm_kahler} and their proofs here, and for his comments on this text.


\begin{thebibliography}{dlRG25}

\bibitem[Bar15]{barlet:hal-01175721}
\bgroup\scshape{}D.~Barlet\egroup{}, \emph{Feuilletage {\`a} feuilles
  compactes}, unpublished, 2015. Available at
  \url{https://hal.science/hal-01175721}.


\bibitem[BB70]{baum-bott}
\bgroup\scshape{}P. Baum\egroup{} and \bgroup\scshape{}R.~Bott\egroup{}, On
  the zeros of meromorphic vector-fields,  in \emph{Essays on {T}opology and
  {R}elated {T}opics ({M}\'{e}moires d\'{e}di\'{e}s \`a {G}eorges de {R}ham)},
  Springer, New York, 1970, pp.~29--47. 
  \doi{10.1007/978-3-642-49197-9_4}.
  
  
\bibitem[BB72]{bb72}
\bgroup\scshape{}P.~Baum\egroup{} and \bgroup\scshape{}R.~Bott\egroup{},
Singularities of holomorphic foliations,  \emph{J. Differential Geometry}
\textbf{7} (1972), 279--342.    Available at
\url{http://projecteuclid.org/euclid.jdg/1214431158}.


\bibitem[BM16]{bogomolov-mcquillan}
\bgroup\scshape{}F.~Bogomolov\egroup{} and
  \bgroup\scshape{}M.~McQuillan\egroup{}, Rational curves on foliated
  varieties,  in \emph{Foliation theory in algebraic geometry}, \emph{Simons
  Symp.}, Springer, Cham, 2016, pp.~21--51.  


\bibitem[BS15]{bracci-suwa}
\bgroup\scshape{}F.~Bracci\egroup{} and \bgroup\scshape{}T.~Suwa\egroup{},
  Perturbation of {B}aum-{B}ott residues,  \emph{Asian J. Math.} \textbf{19}
  no.~5 (2015), 871--885.   \doi{10.4310/AJM.2015.v19.n5.a4}.

\bibitem[Bru11]{brunella-ps}
\bgroup\scshape{}M.~Brunella\egroup{}, Uniformisation de feuilletages et
  feuilles enti\`eres,  in \emph{Complex manifolds, foliations and
  uniformization}, \emph{Panor. Synth\`eses} \textbf{34/35}, Soc. Math. France,
  Paris, 2011, pp.~1--52.  

\bibitem[Bru15]{brunella-birational}
\bgroup\scshape{}M.~Brunella\egroup{}, \emph{Birational geometry of
  foliations}, \emph{IMPA Monographs} \textbf{1}, Springer, Cham, 2015.  \doi{10.1007/978-3-319-14310-1}.

\bibitem[CP94]{campana-peternell}
\bgroup\scshape{}F.~Campana\egroup{} and
  \bgroup\scshape{}T.~Peternell\egroup{}, Cycle spaces,  in \emph{Several
  complex variables, {VII}}, \emph{Encyclopaedia Math. Sci.} \textbf{74},
  Springer, Berlin, 1994, pp.~319--349. 
  \doi{10.1007/978-3-662-09873-8\_9}.

\bibitem[CCG00]{CCGDLF}
\bgroup\scshape{}A.~Campillo\egroup{}, \bgroup\scshape{}M.~M.
  Carnicer\egroup{}, and \bgroup\scshape{}J.~Garc\'{\i}a de~la Fuente\egroup{},
  Invariant curves by vector fields on algebraic varieties,  \emph{J. London
  Math. Soc. (2)} \textbf{62} no.~1 (2000), 56--70.  
  \doi{10.1112/S0024610700008978}.


\bibitem[Can93]{candel}
\bgroup\scshape{}A.~Candel\egroup{}, Uniformization of surface laminations,
  \emph{Ann. Sci. \'{E}cole Norm. Sup. (4)} \textbf{26} no.~4 (1993), 489--516.
   \doi{10.24033/asens.1678}.



\bibitem[CGM95]{candel-gomezmont}
\bgroup\scshape{}A.~Candel\egroup{} and
  \bgroup\scshape{}X.~G\'{o}mez-Mont\egroup{}, Uniformization of the leaves of
  a rational vector field,  \emph{Ann. Inst. Fourier (Grenoble)} \textbf{45}
  no.~4 (1995), 1123--1133.  
  \doi{https://doi.org/10.5802/aif.1488}.

\bibitem[CP06]{coutinho-pereira}
\bgroup\scshape{}S.~C. Coutinho\egroup{} and \bgroup\scshape{}J.~V.
  Pereira\egroup{}, On the density of algebraic foliations without algebraic
  invariant sets,  \emph{J. Reine Angew. Math.} \textbf{594} (2006), 117--135.  \doi{10.1515/CRELLE.2006.037}.
  
\bibitem[dlRG25]{delarosa-parameter}
\bgroup\scshape{}D.~de~la Rosa~G\'omez\egroup{}, Parameter constraints and real
  structures in quadratic semicomplete vector fields on \(\mathbb{C}^3\),
  \emph{Int. J. Math. Math. Sci.} \textbf{2025} (2025), 7371818.
  \doi{10.1155/ijmm/7371818}.

\bibitem[DG23]{DG}
\bgroup\scshape{}B.~Deroin\egroup{} and \bgroup\scshape{}A.~Guillot\egroup{},
Foliated affine and projective structures,  \emph{Compos. Math.} \textbf{159}
no.~6 (2023), 1153--1187.    \doi{10.1112/s0010437x2300711x}.

  
\bibitem[Fis76]{fischer}
\bgroup\scshape{}G.~Fischer\egroup{}, \emph{Complex analytic geometry},
  \emph{Lecture Notes in Mathematics} \textbf{Vol. 538}, Springer-Verlag,
  Berlin-New York, 1976.  



\bibitem[Ghy00]{ghys-jouanolou}
\bgroup\scshape{}E.~Ghys\egroup{}, \`{A} propos d'un th\'{e}or\`eme de
  {J}.-{P}. {J}ouanolou concernant les feuilles ferm\'{e}es des feuilletages
  holomorphes,  \emph{Rend. Circ. Mat. Palermo (2)} \textbf{49} no.~1 (2000),
  175--180.    \doi{10.1007/BF02904228}.

\bibitem[Glu02]{Glutsyuk-simultaneous}
\bgroup\scshape{}A.~Glutsyuk\egroup{}, On simultaneous uniformization and local
  nonuniformizability,  \emph{C. R. Math. Acad. Sci. Paris} \textbf{334} no.~6
  (2002), 489--494.   \doi{10.1016/S1631-073X(02)02268-9}.

\bibitem[GM89]{gomezmont-integrals}
\bgroup\scshape{}X.~G\'{o}mez-Mont\egroup{}, Integrals for holomorphic
  foliations with singularities having all leaves compact,  \emph{Ann. Inst.
  Fourier (Grenoble)} \textbf{39} no.~2 (1989), 451--458.  
  \doi{10.5802/aif.1173}.

\bibitem[GR84]{grauert-remmert}
\bgroup\scshape{}H.~Grauert\egroup{} and \bgroup\scshape{}R.~Remmert\egroup{},
  \emph{Coherent analytic sheaves}, \emph{Grundlehren der mathematischen
  Wissenschaften}
  \textbf{265}, Springer-Verlag, Berlin, 1984.  
  \doi{10.1007/978-3-642-69582-7}.

\bibitem[Gui06]{guillot-fourier}
\bgroup\scshape{}A.~Guillot\egroup{}, Semicompleteness of homogeneous quadratic
  vector fields,  \emph{Ann. Inst. Fourier (Grenoble)} \textbf{56} no.~5
  (2006), 1583--1615.   \doi{10.5802/aif.2221}.

\bibitem[Gui07]{guillot-ihes}
\bgroup\scshape{}A.~Guillot\egroup{}, Sur les \'{e}quations d'{H}alphen et les
  actions de {${\rm SL}_2({\bf C})$},  \emph{Publ. Math. Inst. Hautes
  \'{E}tudes Sci.} no.~105 (2007), 221--294. 
  \doi{10.1007/s10240-007-0008-6}.

\bibitem[Gui18]{guillot-sigma}
\bgroup\scshape{}A.~Guillot\egroup{}, Quadratic differential equations in three
  variables without multivalued solutions: {P}art {I},  \emph{SIGMA Symmetry
  Integrability Geom. Methods Appl.} \textbf{14} (2018), Paper No. 122.
   \doi{10.3842/SIGMA.2018.122}.

\bibitem[Gui24]{guillot-survey}
\bgroup\scshape{}A.~Guillot\egroup{}, On the singularities of complete
  holomorphic vector fields in dimension two,  in \emph{Handbook of geometry
  and topology of singularities {VI}. {F}oliations}, Springer, Cham, 2024,
  pp.~1--37.   \doi{10.1007/978-3-031-54172-8\_1}.

\bibitem[GR12]{guillot-rebelo}
\bgroup\scshape{}A.~Guillot\egroup{} and \bgroup\scshape{}J.~Rebelo\egroup{},
  Semicomplete meromorphic vector fields on complex surfaces,  \emph{J. Reine
  Angew. Math.} \textbf{667} (2012), 27--65. 
  \doi{10.1515/crelle.2011.127}.

\bibitem[Hir73]{hirzebruch-hilbert}
\bgroup\scshape{}F.~E.~P. Hirzebruch\egroup{}, Hilbert modular surfaces,
  \emph{Enseign. Math. (2)} \textbf{19} (1973), 183--281.  
  
\bibitem[Ily24]{Ilyashenko-handbook}
\bgroup\scshape{}Y.~Ilyashenko\egroup{}, Persistence, uniformizanion and
  holonomy,  in \emph{Handbook of geometry and topology of singularities {V}.
  {F}oliations}, Springer, Cham, 2024, pp.~71--97.  
  \doi{10.1007/978-3-031-52481-3\_2}.

\bibitem[Inc44]{ince}
\bgroup\scshape{}E.~L. Ince\egroup{}, \emph{Ordinary {D}ifferential
  {E}quations}, Dover Publications, New York, 1944.  


\bibitem[Kob87]{kobayashi-vectorbundles}
\bgroup\scshape{}S.~Kobayashi\egroup{}, \emph{Differential geometry of complex
  vector bundles}, \emph{Publications of the Mathematical Society of Japan}
  \textbf{15}, Princeton University Press, Princeton, NJ, 1987, Kan\^o{}
  Memorial Lectures, 5.   \doi{10.1515/9781400858682}.

\bibitem[KO70]{kobayashi-ochiai-positive}
\bgroup\scshape{}S.~Kobayashi\egroup{} and \bgroup\scshape{}T.~Ochiai\egroup{},
  On complex manifolds with positive tangent bundles,  \emph{J. Math. Soc.
  Japan} \textbf{22} (1970), 499--525. 
  \doi{10.2969/jmsj/02240499}.


\bibitem[LN94]{lins-simultaneous}
\bgroup\scshape{}A.~Lins~Neto\egroup{}, Simultaneous uniformization for the
  leaves of projective foliations by curves,  \emph{Bol. Soc. Brasil. Mat.
  (N.S.)} \textbf{25} no.~2 (1994), 181--206.  
  \doi{10.1007/BF01321307}.


\bibitem[LNS96]{linsneto-soares}
\bgroup\scshape{}A.~Lins~Neto\egroup{} and \bgroup\scshape{}M.~G.
  Soares\egroup{}, Algebraic solutions of one-dimensional foliations,  \emph{J.
  Differential Geom.} \textbf{43} no.~3 (1996), 652--673.  
  \doi{10.4310/jdg/1214458327}.


\bibitem[Mok12]{mok}
\bgroup\scshape{}N.~Mok\egroup{}, Projective algebraicity of minimal
  compactifications of complex-hyperbolic space forms of finite volume,  in
  \emph{Perspectives in analysis, geometry, and topology}, \emph{Progr. Math.}
  \textbf{296}, Birkh\"{a}user/Springer, New York, 2012, pp.~331--354.  \doi{10.1007/978-0-8176-8277-4\_14}.

\bibitem[MT15]{moller-toledo}
\bgroup\scshape{}M.~M\"oller\egroup{} and \bgroup\scshape{}D.~Toledo\egroup{},
  Bounded negativity of self-intersection numbers of {S}himura curves in
  {S}himura surfaces,  \emph{Algebra Number Theory} \textbf{9} no.~4 (2015),
  897--912.   \doi{10.2140/ant.2015.9.897}.

\bibitem[MZ55]{montgomery-zippin}
\bgroup\scshape{}D.~Montgomery\egroup{} and
  \bgroup\scshape{}L.~Zippin\egroup{}, \emph{Topological transformation
  groups}, Interscience Publishers, New York-London, 1955. 

\bibitem[Per01]{pereira-global}
\bgroup\scshape{}J.~V. Pereira\egroup{}, Global stability for holomorphic
  foliations on {K}aehler manifolds,  \emph{Qual. Theory Dyn. Syst.} \textbf{2}
  no.~2 (2001), 381--384.   \doi{10.1007/BF02969347}.

\bibitem[Reb96]{rebelo-singularites}
\bgroup\scshape{}J.~C. Rebelo\egroup{}, Singularit\'{e}s des flots holomorphes,
   \emph{Ann. Inst. Fourier (Grenoble)} \textbf{46} no.~2 (1996), 411--428.
  \doi{10.5802/aif.1519}.

\bibitem[Reb99]{rebelo-nonisolees}
\bgroup\scshape{}J.~C. Rebelo\egroup{}, Champs complets avec singularit\'{e}s
  non isol\'{e}es sur les surfaces complexes,  \emph{Bol. Soc. Mat. Mexicana
  (3)} \textbf{5} no.~2 (1999), 359--395.



\bibitem[Ver87]{verjovsky}
\bgroup\scshape{}A.~Verjovsky\egroup{}, A uniformization theorem for
  holomorphic foliations,  in \emph{The {L}efschetz centennial conference,
  {P}art {III} ({M}exico {C}ity, 1984)}, \emph{Contemp. Math.} \textbf{58},
  Amer. Math. Soc., Providence, RI, 1987, pp.~233--253.  

\end{thebibliography}

\providecommand{\noopsort}[1]{}
\providecommand{\doi}[1]{\url{https://doi.org/#1}}
\providecommand{\href}[2]{#2}

\end{document}